\documentclass[oneside, 10pt]{amsart}
\usepackage[osf,sc]{mathpazo}
\usepackage[letterpaper, body={13.6cm, 22.4cm}, mag=1000]{geometry}
\usepackage{amsmath, amsthm, amssymb, enumerate, tikz, hyperref}
\usepackage{textcomp}
\usepackage[mathscr]{eucal}
\usepackage{amscd}
\usepackage{tikz}
\usepackage{hyperref}

\usepackage{enumerate}

\theoremstyle{plain}

\numberwithin{equation}{section}
\theoremstyle{plain}
\newtheorem{cor}[equation]{Corollary}
\newtheorem{lemma}[equation]{Lemma}

\newtheorem{proposition}[equation]{Proposition}

\newtheorem{thm}[equation]{Theorem}

\theoremstyle{definition}

\newcommand{\C}{\operatorname{C} }

\newcommand{\Z}{\operatorname{Z} }

\newcommand{\gen}[1]{\left < #1 \right >}

\begin{document}

\title[Finite $p$-groups of nilpotency class $3$ with two conjugacy class sizes]
{Finite $p$-groups of nilpotency class $3$ with two conjugacy class sizes}
\author{Tushar Kanta Naik}
\address{Harish-Chandra Research Institute, HBNI \\
         Chhatnag Road, Jhunsi,
          Allahabad-211 019 \\
                India}
\email{mathematics67@gmail.com {\rm and } tusharkanta@hri.res.in}

\author{Rahul Dattatraya Kitture}
\address{Harish-Chandra Research Institute, HBNI \\
         Chhatnag Road, Jhunsi,
          Allahabad-211 019 \\
                India}
\email{rahul.kitture@gmail.com {\rm and } rahuldattatraya@hri.res.in}

\author{Manoj K. Yadav}
\address{Harish-Chandra Research Institute, HBNI \\
         Chhatnag Road, Jhunsi,
          Allahabad-211 019 \\
                India}
\email{myadav@hri.res.in}


\begin{abstract}
It is proved  that, for a  prime $p>2$ and an integer $n \ge 1$, finite
$p$-groups of nilpotency class $3$ and having only two conjugacy class sizes $1$
and $p^n$ exist if and only if $n$ is even; moreover, for a given even positive
integer, such a group  is unique up to isoclinism (in the sense of Philip Hall).
\end{abstract}
\maketitle

\section{Introduction}
A finite group $G$ is said to be of {\it conjugate type} $(1=m_1,m_2,\cdots,
m_r)$, $m_i<m_{i+1}$, if $m_i$'s are precisely the different sizes of conjugacy
classes of $G$. In this paper we restrict our attention on finite groups of
conjugate type $(1, m)$. Investigation on such groups  was initiated by N.~Ito
\cite{Ito53} in 1953.
He proved that groups of conjugate type $(1,m)$ are nilpotent,
with $m$ a prime power, say $p^n$. In particular, $G$ is a direct product of its
Sylow-$p$ subgroup and some abelian $p'$-subgroup. So, it is sufficient to study
finite  $p$-{\it groups} of conjugate type $(1, p^n)$ for $p$ a prime and $n \ge
1$ an integer. It was proved by K. Ishikawa \cite{Ishikawa2002} that the
nilpotency class of such
finite $p$-groups is either $2$ or $3$.

It follows (see Section 2) that any two isoclinic groups are of same conjugate
type. The study of finite $p$-groups of conjugate type $(1, p^n)$, up to
isoclinism,  was initiated by  Ishikawa \cite{Ishikawa1999}. He classified
such groups for $n \le 2$. As a consequence, it follows that there is no finite
$p$-group of nilpotency class $3$ and conjugate type $(1, p)$, and there is a
unique finite $p$-group, up to isoclinism, of nilpotency class $3$ and conjugate
type $(1,
p^2)$.
The classification, up to isoclinism, of finite $p$-groups of conjugate type
$(1, p^3)$ was recently done in \cite{NY16}, wherein, among other results, it
was observed that there is no finite $p$-group of nilpotency class $3$ and
of conjugate type $(1, p^3)$.

The examples of $p$-group of nilpotency class $3$ and of conjugate type
$(1,p^{2m})$ are known  for all $m \ge 1$. These examples  appeared in the
construction of certain Camina $p$-groups of nilpotency class $3$ by Dark and
Scoppola \cite[p. 796-797]{Dark-Scoppola}. It can be shown that for a  given
integer $m \ge 1$ and a prime $p>2$, the
$p$-group of conjugate type $(1,p^{2m})$ and class $3$, constructed by Dark and
Scoppola, is isomorphic to $\mathcal{H}_m/\Z(\mathcal{H}_m)$, where
$\mathcal{H}_m$ is presented as follows (see Section
\ref{example} for more details):

\begin{equation}\label{unique-grp}
\mathcal{H}_m=
\begin{Bmatrix}
\begin{bmatrix}
1 & 0    & 0 & 0 & 0 \\
a & 1    & 0 & 0 & 0 \\
c & b    & 1 & 0 & 0 \\
d & ab-c & a & 1 & 0 \\
f & e    & c & b & 1 \\
\end{bmatrix}: a,b,c,d,e,f\in\mathbb{F}_{p^m}\end{Bmatrix}.
\end{equation}

In view of these examples, the question asked in  \cite{NY16} reduces to the
following:
Does there exist a finite $p$-group of nilpotency class $3$ and conjugate type
$(1, p^n)$, for an odd prime $p$ and odd integer $n \geq 5$?

We answer this question, by proving the following much general result.
\vspace{3mm}

\noindent{\bf Main Theorem.}
{\it Let $p>2$ be a  prime and $n \ge 1$ an integer. Then there exist
finite
$p$-groups of nilpotency class
$3$ and conjugate type
$(1, p^n)$  if and only if $n$ is even. For each
positive even integer $n = 2m$, every finite  $p$-group  of  nilpotency class
$3$ and of
conjugate type
$(1,p^n)$ is isoclinic to the group
$\mathcal{H}_m/\Z(\mathcal{H}_m)$, where $\mathcal{H}_m$ is as  in
\eqref{unique-grp}.}

\vspace{2mm}

We set some notations for a multiplicatively written finite group $G$ which
are mostly standard. We denote the commutator subgroup of $G$ by
$G'$ and  the center of $G$ by $\Z(G)$. The third term of the lower central
series of $G$ is denoted by $\gamma_3(G)$. To say that $H$ is a
subgroup  of $G$, we write  $H \leq G$.  For the elements $x, y, z \in G$, the
commutator  $[x, y]$ of $x$
and $y$ is defined by  $x^{-1}y^{-1}xy$, and $[x,y,z]$ = $[[x,y],z]$. For an
element $x \in G$, $x^G$ denotes the conjugacy class of $x$ in $G$ and $[G, x]$
denotes the set $\{[g,x] \mid g \in G\}$. Note that if $[G,x]  \subseteq \Z(G)$,
then  $[G, x]$ is a subgroup of $G$. For a subgroup $H$ of $G$ and an element $x
\in G$, by $\C_H(x)$ we denote the centralizer of $x $ in $H$.
The exponent of $G$ is denoted by $exp(G)$.
If $N$ is a normal subgroup of $G$, then the fact that $xy^{-1} \in N$ will be
denoted by  $x \equiv y \pmod{N}$.
By $\mathbb{F}_p$ we denote the field of {\it integers modulo $p$}.

\section{Reductions}

In $1940$, P.~Hall~\cite{Hall40} introduced the concept of  isoclinism among
groups.
Let $X$ be a finite group and $\overline{X} = X/\Z(X)$. Then
commutation in $X$ gives a well defined map $a_{X} : \overline{X} \times
\overline{X} \mapsto X'$ such that $a_{X}(x\Z(X), y\Z(X)) = [x,y]$ for $(x,y)
\in
X \times X$. Two finite groups $G$ and $H$ are said to be  \emph{isoclinic}, if
there
exists an  isomorphism $\phi$ of the factor group $\overline G = G  / \Z(G)$
onto
$\overline{H} = H/\Z(H)$, and an isomorphism
$\theta$ of the subgroup $G'$
onto  $H'$ such that the
following diagram is commutative:
\[
 \begin{CD}
   \overline G \times \overline G  @>a_G>> G'\\
   @V{\phi\times\phi}VV        @VV{\theta}V\\
   \overline H \times \overline H @>a_H>> H'.
  \end{CD}
\]

Note that isoclinism is an equivalence relation among groups. Equivalence
classes under this relation are called
{\it isoclinism families}.
We recall the following two results of  Hall.

\begin{proposition}~\cite[p. 136]{Hall40}\label{prop1}
Let $G$ and $H$ be two isoclinic finite groups. Then $G$ and $H$ are of the same
conjugate type.
\end{proposition}

\begin{proposition}~\cite[p. 135]{Hall40}\label{prop2}
Let $G$ be a finite group. Then there exists a finite group $H$ in the
isoclinism family of $G$ such that $\Z(H) \leq H'$.
\end{proposition}

The following interesting result is due to M. Isaacs.

\begin{thm} \cite[p. 501]{Isaacs1970}\label{prop5}
Let $G$ be a finite group, which contains a proper normal subgroup $N$ such that
all of the conjugacy classes of $G$ which lie outside of $N$ have the same
sizes. Then either $G/N$ is cyclic or every non-identity element of $G/N$ has
prime order.
\end{thm}

As an immediate consequence of the preceding result, we get

\begin{cor}\label{cor1}
For a prime $p$ and an integer $n\ge 1$, let
$G$ be a finite $p$-group of conjugate type $(1, p^n)$. Then $exp(G/\Z(G)) =
p$. In particular, if $G$ is of nilpotency class $>2$, then $p>2$.
\end{cor}

The following result is due to Ishikawa.
\begin{thm}\cite[Main Theorem]{Ishikawa2002}\label{ishikawa2}
For a prime $p$ and an integer $n\ge 1$,
let $G$ be a finite $p$-group of
conjugate type $(1, p^n)$. Then the nilpotency class of $G$ is at most $3$. As a
consequence, $G'$ is elementary abelian.
\end{thm}

The preceding results reduce our study to the finite $p$-groups $G$ satisfying
the following conditions:
\begin{enumerate}
\item $G$ is of conjugate type $(1, p^n)$, $n \ge 2$.
\item Nilpotency class of $G$ is $3$.
\item  $\Z(G) \le G'$.
\item $p>2$.
\end{enumerate}

 For notational convenience, we set

\noindent {\bf Hypothesis (A1).} We say that a finite $p$-group $G$ satisfies
\emph{Hypothesis
(A1)}, if (1)-(4) above hold for $G$.

\section{Key results}

 In this section we determine some  important invariants associated to a finite
$p$-group satisfying Hypothesis (A1).

\begin{lemma}\label{lem2}
Let $G$ satisfy Hypothesis (A1). Then $[G':\Z(G)]<p^n$.
\end{lemma}
\begin{proof}
Let $[G:G']=p^m$ for some integer $m\geq 1$. Since the nilpotency class of $G$
is $3$,
$G'$ is abelian; hence by the given hypothesis, for any $y \in G' \setminus
\Z(G)$, we have
$$p^n = [G:\C_G(y)] \le [G:G']=p^m;$$
hence   $n \le m$.

We proceed by the way of contradiction.
Contrarily  suppose that $[G' : \Z(G)]=p^k$ with $n \le  k$.
 Our plan  is to count the
cardinality of following set in two different ways:
\[
X=\Big{\{}\big(\langle xG'\rangle, \langle y\Z(G)\rangle \big) \mid x \in G
\setminus G', \;y \in G'\setminus \Z(G), \;[x,y]=1\Big{\}}.
\]
Note that $X$ is well defined. For, if $\langle xG'\rangle=\langle x_1G'\rangle$
and $\langle
y\Z(G)\rangle=\langle y_1\Z(G)\rangle$,  then $[x,y]=1$ if and only if
$[x_1,y_1]=1$. Note that $G/\Z(G)$ is of exponent $p$ (by Corollary \ref{cor1})
and  $\Z(G)\le G'$. Thus
$$exp(G/G') = exp(G'/\Z(G))=p.$$
It follows that there are $(p^k-1)/(p-1)$ subgroups of order $p$
in
$G'/\Z(G)$. Since $G'$ is abelian, for each $y \in G' \setminus \Z(G)$,
$G'\subseteq C_G(y)$ and $[\C_G(y):G']=p^{m-n}$. As there are $(p^{m-n} -
1)/(p-1)$ subgroups of order $p$ in  $\C_G(y)/G'$, we get
$$|X|=\frac{(p^{m-n} - 1)(p^k-1)}{(p-1)^2}.$$

On the other hand, fix $x \in G \setminus G'$. Set  $\C_{G'}(x) := \C_G(x) \cap
G' $. If $[G':\C_{G'}(x)] \ge  p^n$, then there are at least $p^n$ conjugates of
$x$ in $G'$;   hence, by the given hypothesis, there are exactly $p^n$
conjugates of $x$
in $G'$. Consequently we have $x^G=x^{G'}$. Let $\{x_1,\ldots, x_m\}$ be
a minimal generating set of $G$. Then $x^{x_i} = x^{h_i}$ for some
$h_i\in G'$; hence $x^{x_ih_i^{-1}}=x$ for $1 \le i \le m$. Hence
$\{x_1h_1^{-1},\ldots,
x_mh_m^{-1}\}$ is also a generating set for $G$, which centralizes $x$, which
implies that $x \in \Z(G)$,
a contradiction. Thus  $[G':\C_{G'}(x)] \leq p^{n-1}$ and  $[\C_{G'}(x): \Z(G)]
\geq p^{k+1-n} \geq p$. Thus, there are at least $(p^{k+1-n}- 1)/(p-1)$
subgroups $\langle y\Z(G)\rangle$ of order $p$ in $G'/\Z(G)$ with $[x,y]=1$.
Counting all together, we  get
$$|X| \geq \frac{(p^m-1)(p^{k+1-n}- 1)}{(p-1)^2}.$$
Comparing the size of $X$, we get
$$(p^m - 1)(p^{k+1-n}- 1) \leq (p^{m-n} - 1)(p^k- 1),$$
which on simplification gives
$$p+p^{-k}+p^{n-m} \leq 1+p^{1-m}+p^{n-k}<3,$$
 a contradiction on the choice of $p$. Hence  $[G' : \Z(G)] < p^n$.
\end{proof}

Before proceeding further, we introduce the notion of breadth in $p$-groups. Let
$G$ be a finite $p$-group. For $x \in G$, the \emph{breadth $b(x)$} of $x$ in
$G$ is defined as $p^{b(x)} = [G : \C_G(x)]$. The \emph{breadth $b_G$} of G is
defined as
$$b_G= \mbox{max}\{b(x) \mid x \in G \}.$$

Let $A$ be an abelian normal subgroup of $G$. Then we define the following:
\begin{align*}
p^{b_A(x)} &= [A : C_A(x)].\\
b_A(G) &= \mbox{max}\{ b_A(x) \mid x \in G \}. \\
B_A(G) &= \{x \in G \mid b_A(x) = b_A(G)\}.
\end{align*}

For the ease of notation, we denote $B_A(G)$ by $B_A$.

\begin{lemma}\label{lem3}
Let $G$ satisfiy Hypothesis (A1). Then $$B_{G'}=\{x \in G \mid
\C_{G'}(x)=\Z(G)\}$$ and $\langle B_{G'}\rangle = G$.
\end{lemma}

\begin{proof}
Let $[G:G']=p^m$ for some integer $m$. As in the proof of Lemma \ref{lem2}, we
have  $n \le m$. Let $[G':\Z(G)]=p^k$. Then, by  Lemma \ref{lem2},  $k \le
n-1$.
Define
$$T:=\{x \in G \mid x \mbox{ commutes with some element } y \in G'\setminus
\Z(G)\}.$$
Note that $T= \cup \C_G(h)$, where union is taken over all subgroups  $\langle
hG'\rangle$ of order $p$ in $G'/\Z(G)$.
Since $G'$ is abelian, for each $h
\in G'\setminus \Z(G)$, $G' \le \C_G(h)$ and $|\C_G(h)|=|G'|\,p^{m-n}$.
By Corollary \ref{cor1}, $exp(G/\Z(G))=p$, and hence $exp(G'/\Z(G))=p$. Thus,
\begin{equation*}
|T| \,\leq\, |G'|\,p^{m-n}\frac{(p^k-1)}{p-1} \,<\, |G'|\,p^{m-n+k} \,\leq\,
|G'|\,p^{m-1} \,< \, |G|.
\end{equation*}
Consequently, $T$ is a proper subset of $G$, and therefore there exists an
element $x \in G \setminus G'$ such that $\C_{G'}(x)=\Z(G)$. Thus
$$B_{G'}=\{x\in G \mid \C_{G'}(x)=\Z(G)\} = G \setminus T.$$
If $\langle B_{G'}\rangle=H<G$, then $|H| \leq |G'|\,p^{m-1}$. Thus we have
\begin{align*}
|G| = |T \cup B_{G'}|\leq |G'|p^{m-1} + |G'|\,p^{m-1} < |G'|p^m=|G|,
\end{align*}
which is a contradiction. Hence  $\langle B_{G'} \rangle = G$, which completes
the proof.
\end{proof}

As a consequence of Hall-Witt identity, we get
\begin{lemma}\label{lem1}
Let $G$ be a group of nilpotency class $3$ and let $x,y,z \in $ $G$ be such
that $[x,z], [y,z]
\in $ $\Z(G)$.  Then $[x,y,z] = 1$, that is, $[x,y] \in G'  \cap \C_G(z)$.
\end{lemma}

The following  result computes the index of $G'$ in $G$.

\begin{lemma}\label{lem4}
If $G$ satisfy Hypothesis (A1), then $[G:G']=p^n$.
\end{lemma}

\begin{proof}
For $n=2$, the result is proved in  \cite [Theorem 4.2]{Ishikawa1999}. So we
assume that $n>2$. Since $G'$ is abelian, it follows that  $[G:G'] \geq
p^n$. If $[G:G']=p^n$, then there is nothing to prove. So assume that
$[G:G']=p^{n+m}$ with $m>0$.  We now proceed to get a contradiction. Let
$[G':\Z(G)] = p^k$, where $k \leq n-1$ (by Lemma \ref{lem2}). We complete the
proof in several steps.

\vskip3mm
\noindent {\bf Step 1.} {\it For all $h \in G' \setminus \Z(G)$,
$\C_G(h)=\C_G(G')$.}
\vskip2mm
It follows from Lemma \ref{lem3} that we can choose an element $x_1 \in G$ such
that
$$\C_{G'}(x_1)=\C_{G}(x_1) \cap G' =\Z(G).$$
 Let $\{x_1, x_2, \ldots,x_{m+k},\ldots, x_{m+n}\}$ be
a minimal generating set for $G$ such that $\C_G(x_1)=\langle x_1, x_2, \dots,
x_{m+k}, \Z(G)\rangle$.

If $\{[x_1,x_i] \mid m+k< i\leq m+n \}\subseteq \Z(G)$, then $[x_1,x_i] \in
\Z(G)$ for all $i$ with $1 \le i \le m+n$. By Lemma \ref{lem1}, we get
$[x_r,x_s] \in \C_G(x_1) \cap G' =  \Z(G)$ for all $r, s$ with $1\le r,s\le
m+n$, that is, $G' \subseteq
\Z(G)$,
a contradiction. Thus, without loss of generality, we can assume that
$[x_1,x_{m+n}] \notin \Z(G)$. Consider the  subgroup
$$K=\langle [x_1,x_{m+n}], [x_2,x_{m+n}],\ldots, [x_{m+k}, x_{m+n}],
\Z(G)\rangle.$$
Since $[G':\Z(G)]=p^k$, among the commutators $[x_i,x_{m+n}]$, $1\leq i\leq
m+k$, at most $k$ are independent  modulo $\Z(G)$ (viewed as elements of the
vector space $G'/\Z(G)$).
We can assume,  without loss of generality, that for some integer $l$ with $1
\le l \le k$,
$[x_1,x_{m+n}], [x_2,x_{m+n}],\ldots, [x_{l}, x_{m+n}]$
are independent elements modulo $\Z(G)$, and generate $K$ along  with $\Z(G)$.

We now show that $l=k$. For any $t$ with $l<t\leq m+k$, we have
$$[x_t,x_{m+n}]\equiv [x_1, x_{m+n}]^{i_1} [x_2, x_{m+n}]^{i_2} \cdots [x_l,
x_{m+n}]^{i_l}  \pmod{\Z(G)}.$$
Let $x_t'=x_tx_1^{-i_1} \cdots x_l^{-i_l}$. Then $[x_t',x_{m+n}] \in \Z(G)$.
Thus,  with $\C_G(x_1)=\langle x_1$, $x_2$, $\ldots$, $x_{m+k}$, $\Z(G)\rangle$,
we can assume, modifying $x_t$ by $x_t'$ (if necessary), that
\begin{equation}\label{eqse3-a}
[x_t,x_{m+n}] \in \Z(G) \mbox{ for } l<t\leq m+k.
\end{equation}
Since $[x_1,x_i]=[x_1,x_j]=1$ for any $i,j\leq m+k$, by Lemma \ref{lem1},
we get
$$[x_j,x_i]\in \C_G(x_1)\cap G'=\Z(G).$$ In particular, for $l<t\leq m+k$,
$[x_t,x_i]\in \Z(G)$. Consequently by \eqref{eqse3-a} and  Lemma \ref{lem1},
$$[x_t,[x_i,x_{m+n}]]=1  \hskip5mm (l<t\leq m+k, \,\,1 \le i\leq l).$$
Thus for any $i$ with $1 \le i \leq l$, $\langle
x_{l+1},\ldots, x_{m+k}, G'\rangle \le \C_G([x_i,x_{m+n}])$; so
$$p^n=[G:\C_G([x_i, x_{m+n}])] \leq p^{m+n-(m+k-l)} =p^{n-k+l}\leq p^n.$$
Consequently, we get
\begin{enumerate}[(i)]
\item $k=l$,
\item $K=\langle [x_1,x_{m+n}], [x_2,x_{m+n}],\ldots, [x_k,x_{m+n}],
\Z(G)\rangle=G'$, and
\item $\C_G([x_i,x_{m+n}]) =\langle  x_{k+1},x_{k+2}\ldots,x_{k+m}, G'\rangle :=
H$ (say), for all $i$, $1 \le i\leq k$.
\end{enumerate}
By (ii) and (iii), we get  $\C_G(h)=H = \C_G(G')$ for all $h\in
G'\setminus\Z(G)$. This proves  Step 1.

Observe that, by Lemma \ref{lem3},
$$G \setminus H=\{x \in G \mid \C_{G'}(x)=\Z(G)\}=B_{G'},$$ which implies  that
$\C_H(x) \cap G' = \Z(G)$ for all $x \in G \setminus H$.
\vskip3mm

\noindent{\bf  Step 2.} {\it For any $y \in G \setminus H$, there exist elements
$h_1,h_2,\ldots,h_m \in G'$ such that
 $$\langle x_{k+i}h_i: 1\leq i \leq m \rangle\Z(G)/\Z(G) = (\C_G(y)\cap
H)/\Z(G)$$
 and are of order $p^m$.}

\vskip2mm
Let $y \in G \setminus H$ be an arbitrary element. Then  $\C_{G'}(y)=\Z(G)$.
Consider a minimal generating set  $Y:= \{y =y_1,y_2,\ldots,y_{m+k},\ldots,
y_{m+n}\}$ of $G$ such that $\C_G(y_1)=\langle y_1,y_2,\dots,y_{m+k},
\Z(G)\rangle.$
As in Step 1, we can modify (if necessary) the elements $y_{k+1}, \dots,$
$y_{m+k}$ so that
$$\langle y_{k+1},y_{k+2},\dots,y_{m+k}, G'\rangle = H =\langle
x_{k+1},x_{k+2},\dots,x_{m+k}, G' \rangle.$$

Consequently, for $ 1 \le i \le m$, it follows that $x_{k+i} = (\prod_j
y_{k+j}^{a_{ij}})h_i^{-1}$ for some integers $a_{ij}$ and some element $h_i \in
G'$. Hence $x_{k+i}h_i \in \C_G(y)$, which shows that
\begin{equation}\label{eqse3-0a}
\langle x_{k+i}h_i \mid 1\leq i \leq m \rangle\Z(G)/\Z(G) \le (\C_G(y)\cap
H)/\Z(G).
\end{equation}
Note that for $1 \le i, j \le m$, $[x_{k+i}h_i, x_{k+j}h_j] \in \Z(G)$. Thus
\begin{equation}\label{eqse3-0b}
|\langle x_{k+i}h_i \mid 1\leq i \leq m \rangle\Z(G)/\Z(G)| = p^m.
\end{equation}
By the vary choice of $y_{k+1}, \dots,y_{m+k}$, it follows that
$$\gen{y_{k+1}, \dots, y_{k+m}, \Z(G)} =  \C_G(y)\cap H;$$ hence  $|(\C_G(y)\cap
H)/\Z(G)| = p^m$. This, along with  \eqref{eqse3-0a} and \eqref{eqse3-0b},
proves Step 2.

\vskip3mm

\noindent{\bf Step 3.} {\it The cardinality of
$$S := \Big{\{} \big{(} \langle x\Z(G)\rangle, \langle y\Z(G)\rangle\big{)}
\mid x\in G\setminus H,\,\, y\in H\setminus G',\,\, [x,y]=1\Big{\}}$$ is
$p^{m+k}(p^n-1)(p^m-1)/(p-1)^2$.}

Since the exponent of $G/\Z(G)$ is $p$, it follows that $G/\Z(G)$, $H/\Z(G)$ and
$G'/\Z(G)$ have $(p^{n+m+k}-1)/(p-1)$, $(p^{m+k}-1)/(p-1)$ and $(p^k-1)/(p-1)$
number of subgroups of order $p$ respectively. Consequently
\begin{equation}\label{eqse3-1}
|\{\gen{x\Z(G)} \mid x \in G\setminus H\}| = \frac{p^{m+k}(p^n-1)}{p-1}
\end{equation}
and
\begin{equation}\label{eqse3-2}
|\{\gen{y\Z(G)} \mid y \in H\setminus G'\}| = \frac{p^{k}(p^m-1)}{p-1}.
\end{equation}

For each $x \in G \setminus H$, by Step  $2$, we have
$|(\C_G(x) \cap H)/\Z(G)| = p^m$; hence
\begin{equation}\label{eqse3-3}
|\{\gen{y\Z(G)} \mid y \in H \setminus G', [x, y] = 1\}| = \frac{p^m-1}{p-1}.
\end{equation}

 Hence, by \eqref{eqse3-1} and \eqref{eqse3-3}, we have
 \begin{equation}\label{eqse3-4}
 |S|=\frac{p^{m+k}(p^n-1)(p^m-1)}{(p-1)^2},
 \end{equation}
and the proof of Step 3 is complete.

We now proceed to get the final contradiction.
Note that there exists  some element $y_0\in H \setminus G'$ such that
\begin{equation}\label{eqse3-5}
|\{\gen{x\Z(G)} \mid x \in G\setminus H, [x,y_0]=1\}| \ge
\frac{p^{m}(p^n-1)}{p-1}.
\end{equation}
For, if there is no such $y_0$, then for each $y \in H \setminus G'$, we get
$$|\{\gen{x\Z(G)} \mid x \in G\setminus H, [x,y]=1\}| <
\frac{p^{m}(p^n-1)}{p-1}.$$
But then
$$|S| < \frac{p^{m+k}(p^n-1)(p^m-1)}{(p-1)^2},$$
which is absurd.

Note that $G' \le \C_G(y_0)$, $[C_G(y_0):G']=p^m$, $|G/H| = p^n$ and $|H/G'| =
p^m$.
 Assume that
 $$|\C_G(y_0)H/H| = p^{n-s} ~~\text{and} ~~ |\C_H(y_0)/G'| = p^{m-r},$$
 for some integers $s$ and $r$.
 Since $[G: \C_G(y_0)] = p^n$, it follows that $n = s+r$.
 Then
 $$|\C_G(y_0)/\Z(G)| = p^{k+(m-r)+(n-s)} = p^{k+m}$$
 and
 $$|\C_H(y_0)/\Z(G)| = p^{k+(m-r)}.$$
 Thus
 \begin{eqnarray*}
 |\{\gen{x\Z(G)} \mid x \in G \setminus H, [x, y_0] = 1\}| &=&
\frac{(p^{k+m}-1)-(p^{k+m-r}-1)}{p-1}\\
 &=& \frac{p^{m+k-r}(p^{n-s}-1)}{p-1}.
 \end{eqnarray*}

The preceding equation along with \eqref{eqse3-5} gives
$$p^m(p^n-1)/(p-1)  \le p^{m+k-r}(p^{n-s}-1)/(p-1),$$
which implies that $(p^n-1)  \leq  p^k - p^{k-r}.$
Hence, by Lemma \ref{lem2},
$$p^n < p^k +1 \le p^{n-1} + 1,$$
which is absurd, and the proof of the lemma is complete.
\end{proof}

As an immediate consequence of the preceding lemma, we get the following
important information about centralizers of elements in $G$, which we use
frequently without any further reference.
\begin{cor}
If $G$ satisfies Hypothesis (A1), then for all $x \in G \setminus G'$ the
following hold:
\begin{enumerate}
\item $\C_G(x) \cap G'=\Z(G).$
\item $[\C_G(x) : \Z(G)] = [G' : \Z(G)]$.
\end{enumerate}
\end{cor}

\begin{thm}\label{lem6}
Let $G$ satisfies Hypothesis ({A1}). If $[G':\Z(G)]=p^m$, then
$n=2m$.
\end{thm}
\begin{proof}
Consider $x,y\in G$ with $[x,y]\notin \Z(G)$. We consider two cases, namely (1)
$n > 2m$ and (2) $n <2m$, and get contradiction in both.

\textbf{Case 1.} $n > 2m$

Write $\overline{G}=G/\Z(G)$. Since $\overline{G}$ is of nilpotency class $2$
and  $|\overline{G}'|=p^m$, we have
$$[\overline{G}:\C_{\overline{G}}(\overline{x})]\leq p^m \mbox{ and }
[\overline{G}:\C_{\overline{G}}(\overline{y})]\leq p^m.$$
Since $[\overline{G}:\overline{G}']=p^n>p^{2m}$, there exists
$\overline{w} \notin
\overline{G}'$ such that $\overline{w}\in \C_{\overline{G}}(\overline{x})
\cap \C_{\overline{G}}(\overline{y})$, i.e. $[x,w],[y,w]\in \Z(G)$.
Since  $w \notin G'$, by Lemma \ref{lem1}, $[x,y]\in \C_G(w)\cap G'=\Z(G)$,
a contradiction to our supposition that  $[x,y] \notin \Z(G)$.

\textbf{Case 2.} $n < 2m$.

Since $x \not\in G'$, we have
$$[\C_G(x)G':G']=[\C_G(x):\C_G(x)\cap G']= [\C_G(x) : \Z(G)] = p^m,$$
 and similarly $[\C_G(y)G':G']=p^m$.
Since $n<2m$,  we have $[G:G']=p^n<p^{2m}$; hence $\C_G(x)G'\cap \C_G(y)G'$
contains $G'$ properly. Consider $w\in \C_G(x)G'\cap \C_G(y)G'$ with $w\notin
G'$. Since
$G$ is of nilpotency class $3$, it is easy to see that $[w,x], [w,y]\in \Z(G)$,
hence
by Lemma \ref{lem1}, $[x,y]\in \C_G(w)\cap G'=\Z(G)$, a contradiction again.

Hence $n =2m$, and the proof is complete.
\end{proof}

Before proceeding further, we strengthen Hypothesis (A1) as follows:

\vspace{3mm}
\noindent {\bf Hypothesis (A2).} We say that a finite $p$-group $G$ satisfies
\emph{Hypothesis
(A2)}, if $G$ is  of nilpotency class $3$ and of conjugate type
$(1,p^{2m})$ with $\Z(G) \le G'$.
\vspace{3mm}

A group $G$ is said to be a {\it Camina group} if $xG' = x^G$ for all $x \in G
\setminus G'$.
For determining the structure of $G/\Z(G)$ when $G$ satisfies Hypothesis (A2),
the following result  of Verardi \cite{Verardi1987} (also see \cite[Lemma
1.2]{Mann-Scoppola}) will be used.
\vspace{3mm}

\begin{thm}~\cite[Lemma 1.2]{Mann-Scoppola}\label{prop8}
For an odd prime $p$, let $G$ be a Camina $p$-group  of order $p^{3n}$,
of exponent $p$ and of nilpotency class $2$. Let $[G:G']=p^{2n}$ and there are
two elementary abelian subgroups $A^{*}$, $B^{*}$  of $G$ such that
$G=A^{*}B^{*}$, $A^{*}=A \times G'$, $B^{*}=B \times G'$, and thus $G=ABG'$.
Then the following statements are equivalent:
\begin{enumerate}
\item $G$ is isomorphic to ${\rm U}_3(p^n)$.
\item All the centralizers of non-central elements of $G$ are abelian.
\end{enumerate}
\end{thm}

\begin{thm}\label{lem7}
Let $G$ satisfy Hypothesis (A2). Write $\overline{G}=G/\Z(G)$ and
$\overline{x}=x\Z(G)$ for $x\in G$. Then the following hold:
\begin{enumerate}
\item $\overline{\C_G(x)G'}= \C_{\overline{G}}(\overline{x})$ and
$|\C_{\overline{G}}(\overline{x})|=p^{2m}$ for $x\in G\setminus G'$.
\item$\overline{G}$ is a Camina $p$-group of order $p^{3m}$,
exponent $p$, and $|\overline{G}'|=p^m$.
\item All the  centralizers of non-central elements in $\overline{G}$ are
elementary abelian of
order $p^{2m}$.
\item If $A=\C_{\overline{G}}(\overline{x})$ and
$B=\C_{\overline{G}}(\overline{y})$ for $\overline{x}, \overline{y} \not\in
\Z(\overline{G})$  are
distinct centralizers in $\overline{G}$, then
$A\cap B=\Z(\overline{G})$ and $\overline{G} = AB$.
\item $\overline{G}$ is isomorphic to ${\rm U}_3(p^m)$.
\end{enumerate}
\end{thm}
\begin{proof}
$(1)$ By Lemma \ref{lem4} and Theorem \ref{lem6}, $[G:G']=p^{2m}$ and
$[G':\Z(G)]=p^m$; hence
$$|\overline{G}|=p^{3m}~~ \mbox{and}~~|\overline{G}'|=p^m.$$
 Further, for $h\in G'\setminus \Z(G)$, $\C_G(h)=G'$.

Since $\overline{G}$ is of nilpotency class $2$, $\overline{G}'$ is contained in
the centralizer of every element  in $\overline{G}$. Hence
$\overline{\C_G(x)G'}\leq \C_{\overline{G}}(\overline{x})$ for any $x\in
G\setminus G'$. Since $[\C_G(x):\Z(G)]=[G':\Z(G)]=p^m$ with $\C_G(x)\cap
G'=\Z(G)$, we get $[\C_G(x)G':\Z(G)]=p^{2m}$ and $[G:\C_G(x)G']=p^m$. Hence
$[\overline{G}:\C_{\overline{G}}(\overline{x})]\leq p^m$ and
$|\C_{\overline{G}}(\overline{x})|\geq p^{2m}$.

Fix $x\in G\setminus G'$. If possible, suppose that
$|\C_{\overline{G}}(\overline{x})| > p^{2m}$, that is,
$[\overline{G}:\C_{\overline{G}}(\overline{x})]<p^m$. Note that
$\overline{x}\notin \Z(\overline{G})$. For, let $\overline{x}\in
\Z(\overline{G})$. For any minimal generating set $\{x_1,\ldots, x_{2m}\}$
of $G$, $\{\overline{x}_1,\ldots,\overline{x}_{2m}\}$ is also a minimal
generating set for $\overline{G}$. Then $[\overline{x},\overline{x}_i]=1$, that
is $[x,x_i]\in \Z(G)$ for $1\le i\le 2m$. Then by Lemma \ref{lem1}, for $1\le
i\le 2m$, $[x_i,x_j]\in C_G(x)\cap G'=\Z(G)$ (since $x\notin G'$); hence
$G'\subseteq \Z(G)$, a contradiction.

Then there exists $t\in G$ such that $[t,x]\notin \Z(G)$. Since
$|\overline{G}'|=p^{m}$, we have
$[\overline{G}:\C_{\overline{G}}(\overline{t})]
\leq p^{m}$, and therefore it follows that
$\C_{\overline{G}}(\overline{x})\cap \C_{\overline{G}}(\overline{t})$ contains
$\overline{G}'$ properly. Take $\overline{w}\in
\C_{\overline{G}}(\overline{x})\cap \C_{\overline{G}}(\overline{t})\setminus
\overline{G}'$. Then $[w,x], [w,t]\in \Z(G)$ with $w\notin G'$. By Lemma
\ref{lem1}, $[x,t]\in \C_G(w) \cap G'=\Z(G)$, a contradiction. Thus
$|\C_{\overline{G}}(\overline{x})|=p^{2m}=|\overline{\C_G(x)G'}|$. This proves
assertion $(1)$.

$(2)$ By Corollary \ref{cor1}, $exp(\overline{G})=p$. Now the assertion $(2)$
follows from assertion $(1)$.

$(3)$ Consider any $x\in G\setminus G'$. For $y_1,y_2\in \C_G(x)$, by Lemma
\ref{lem1}, $[y_1,y_2]\in \C_G(x)\cap G'=Z(G)$. Hence
$[\C_G(x),\C_G(x)] \leq \Z(G)$. Since $G$ is of nilpotency class $3$, we have
$$[\C_G(x)G', \C_G(x)G'] = [\C_G(x),\C_G(x)]\, [\C_G(x),G'] \, [G',G']\leq
\Z(G),$$
i.e. $\C_{\overline{G}}(\overline{x})=\overline{\C_G(x)G'}$ is abelian. Since
$\overline{G}$ is of exponent $p$, then so is $\C_{\overline{G}}(\overline{x})$.

$(4)$ It is given that $A=\C_{\overline{G}}(\overline{x})$ and
$B=\C_{\overline{G}}(\overline{y})$ are
distinct proper subgroups of  $\overline{G}$, and are abelian by assertion (3).
Thus for any element $\overline{w} \in A \cap B$, it follows that
$|\C_{\overline{G}}(\overline{w})| \ge |AB| > p^{2m}$. Hence, by (3),
$\overline{w} \in \Z(\overline{G})$. Since $|\overline{G}|=p^{3m}$,
$|\Z(\overline{G})|=p^m$ and
$|A|=|B| = p^{2m}$,  we have  $\overline{G} = AB$.

$(5)$ The assertion follows by $(2)$-$(4)$ along with  Theorem \ref{prop8}.
\end{proof}

Using Theorem \ref{lem7}(1), the following result is an easy exercise.
\begin{cor}\label{se3-cor1}
Let $G$ satisfy Hypothesis (A1) and $u, v \in G \setminus G'$ such that $[u,v]
\in \Z(G)$. Then there exists an element $h \in G'$ such that $[u,vh] = 1$.
\end{cor}

As a consequence of Hall-Witt identity, in a $p$-group $G$ of nilpotency class
$3$, one can deduce that
\begin{equation}\label{eqse3-5a}
\mbox{ If } [a,b]\in \Z(G) \mbox{ then } [[a,t],b]=[[b,t],a]
\hskip5mm (a,b,t \in G).
\end{equation}


\begin{lemma}\label{lem9}
Let $G$ satisfy Hypothesis (A2). Then $\Z(G)=\gamma_3(G)$, and is elementary
abelian of order $p^{2m}$.
\end{lemma}
\begin{proof}
Since $G'$ is elementary abelian (by Theorem \ref{ishikawa2}), so are
$\gamma_3(G)$ and  $\Z(G)$ (since $\Z(G)\le G'$).
Consider $x_1\in G\setminus G'$. Recall that  $\C_G(x_1)\cap G'=\Z(G)$, $
[\C_G(x_1) : \Z(G)] = [G' : \Z(G)] = p^m$ and $[G : G'] = p^{2m}$.  Let
$\C_G(x_1)=\langle
x_1,\ldots,x_m,\Z(G)\rangle$. Consider $y_1\in G\setminus \C_G(x_1)G'$. Let
$\C_G(y_1)=\langle y_1,\ldots, y_m,\Z(G)\rangle$.
Then $\C_{\overline{G}}(\overline{x}_1)=\langle
\overline{x}_1,\ldots,\overline{x}_m,\overline{G}'\rangle$ and
$\C_{\overline{G}}(\overline{y}_1)=\langle
\overline{y}_1,\ldots,\overline{y}_m,\overline{G}'\rangle$ are distinct proper
subgroups of $\overline{G}$; hence, by Theorem \ref{lem7}(4), they generate
$\overline{G}$. It follows
that $\{ x_1,\ldots, x_m,y_1,\ldots,y_m\}$ is a (minimal) generating set for
$G$.
Define
\begin{align*}
[x_1,y_i] =h_i,
\hskip5mm 1\le i \le m.
\end{align*}
By Theorem \ref{lem7}(2), $h_1,\ldots, h_m$ are independent modulo
$\Z(G)$, and  $$G'=\langle h_1,\ldots, h_m,\Z(G)\rangle.$$

Next define
$$[h_1,x_i]=z_i,  \mbox{ and }  [h_1,y_i]=z_{m+i}, \hskip5mm 1 \le i \le m.$$
Since $x_1, \dots, x_m, y_1, \dots, y_m$ are independent modulo $G' =
\C_G(h_1)$,
it follows that
$$[h_1,x_1],\ldots, [h_1,x_m], [h_1,y_1],\ldots, [h_1,y_m]$$
 are independent, and they generate a subgroup $K$ of order $p^{2m}$ in
$\gamma_3(G)$.

We now proceed to show that $K=\gamma_3(G)$.
It is sufficient to show that for any $h\in
G'\setminus \Z(G)$, $[h,x_i],[h,y_i]\in K$ for $1 \le i \le m$.
For any $h\in G'\setminus \Z(G)$ and fixed $i$ with $1 \le i \le m$, consider
$[h,x_i]$. Let
$A=\langle x_1,\ldots, x_m,\Z(G)\rangle$ and $B=\langle
y_1,\ldots,y_m,\Z(G)\rangle$. Since $\langle [x_1,y_1], \dots, [x_1,y_m],  \Z(G)
\rangle=G'$,
there exists $y\in B$ such that $h \equiv
[x_1,y]\pmod{\Z(G)}$. Then by \eqref{eqse3-5a}
\begin{align*}[h,x_i] &=[[x_1,y],x_i]=[[x_i,y],x_1].\end{align*}
Again, since $\langle [x_1,y_1],\dots, [x_m,y_1], \Z(G) \rangle=G'$, there
exists
$x\in A$ such that $[x_i,y]\equiv [x,y_1] \pmod{\Z(G)}$. Therefore, by
\eqref{eqse3-5a}
\begin{align*}
[h,x_i]=[[x_i,y],x_1] =[[x,y_1],x_1]&=[[x_1,y_1],x]
=[h_1,x]\in K.
\end{align*}
Similarly we can show that $[h,y_i]\in K$, $1\le i\le m$; hence
$K=\gamma_3(G)$ and is of order $p^{2m}$.

It only remains to show that $\Z(G)=\gamma_3(G)$. For this, since
$[x_1,y_1],\ldots,
[x_1,y_m]$ are independent modulo $\Z(G)$ and
$$G'=\langle [x_1,y_1], \ldots, [x_1,y_m], \Z(G)\rangle,$$
it suffices to prove that
\begin{equation}\label{eqse3-6}
G'=\langle [x_1,y_1], \ldots, [x_1,y_m], \gamma_3(G)\rangle.
\end{equation}
Note that $\gamma_3(G)\subseteq \Z(G)$ and $\gamma_3(G)=\langle z_1,\ldots,
z_{2m}\rangle$. Also
$$A=\langle x_1,\ldots, x_m,\Z(G)\rangle=\C_G(x_1) \mbox{ and } B=\langle
y_1,\ldots,y_m,\Z(G)\rangle=\C_G(y_1).$$

 If \eqref{eqse3-6} does not hold, then there exist $z\in \Z(G)\setminus
\gamma_3(G)$ and a commutator $[x_i, y_j]$ for some $i,j$ with $1\leq i,j\leq
m$,  such that
\begin{equation*}
[x_i,y_j]=[x_1,y_1]^{e_1}\cdots [x_1,y_m]^{e_m}z,
\end{equation*}
where $e_i \in \mathbb{F}_p$ for $1 \le i \le m$.
Let $y=y_1^{e_1}\cdots y_m^{e_m}$. Then the preceding equation implies
\begin{equation}\label{eqse3-7}
[x_i,y_j]\equiv [x_1,y]z\pmod{\gamma_3(G)}.
\end{equation}
Consequently $[x_i,y_j]\equiv [x_1,y]\pmod{\Z(G)}$.  Since $[x_1,x_i]\equiv
[y,y_j]\equiv
1\pmod{\Z(G)}$, it follows that
$[x_1y_j, x_iy]\equiv[x_1,y][y_j,x_i]\equiv 1\pmod{\Z(G)}.$
Hence, by Corollary \ref{se3-cor1}, there exists $h\in G'$ such that $[x_1y_j,
x_iyh]=1$. Since $G/\gamma_3(G)$ is of nilpotency class $2$, using
\eqref{eqse3-7}, we get
\begin{align*}
1=[x_1y_j,
x_iyh] &\equiv [x_1,x_i][x_1,y][x_i,y_j]^{-1}[y_j,y]\pmod{\gamma_3(G)}\\
&\equiv z^{-1}[y_j,y]\pmod{\gamma_3(G)}.
\end{align*}
Since $y_j,y\in \C_G(y_1)$, we have $\overline{y}_i,\overline{y}\in
\C_{\overline{G}}(\overline{y_1})$,
which is abelian by Theorem \ref{lem7}(3); hence
$[\overline{y_j},\overline{y}]=1$, i.e. $[y_j,y]\in
\Z(G)$.
Again by Corollary \ref{se3-cor1}, there exists $h_1\in G'$ such that
$[y_j,yh_1]=1$. Then
$[y_j,y]=[y_j,h_1]^{-1}\in\gamma_3(G)$. Thus, we get
$$1\equiv z^{-1} \pmod{\gamma_3(G)},$$
a contradiction.

This proves that \eqref{eqse3-6} holds. Hence
$\gamma_3(G)=\Z(G)$, and the proof is complete.
\end{proof}


\section{Examples}\label{example}

In this section, we describe the examples of $p$-groups of conjugate type $(1,
p^{2m})$ and class $3$ from the construction by  Dark and Scoppola
\cite{Dark-Scoppola}.

For an odd prime $p$ and an integer $m \ge 1$, let $q=p^m$ and $\mathbb{F}_{q}$
denote the field of
order $q$. Consider the set $\mathcal{G}$ of quintuples $(a,b,c,d,e)$ over
$\mathbb{F}_q$. Define an operation `.' on $\mathcal{G}$ as follows.  For any
two quintuples $(a,b,c,d,e)$ and $(x,y,z,u,v)$, define $(a,b,c,d,e).(x,y,z,u,v)$
to be  the quintuple
\begin{align*}
(a+x,b+y,c+z+bx, d+u+az+(ab-c)x,e+v+cy+b(xy-z).
\end{align*}
A routine  check shows that $\mathcal{G}$ is a group under this operation, with
identity element $(0,0,0,0,0)$, which we denote by ${\mathbf 0}$, and
$$(a,b,c,d,e)^{-1}=(-a,-b,-c+ab,-d,-e).$$
Then $(a,b,c,d,e).(x,y,z,u,v).(a,b,c,d,e)^{-1}.(x,y,z,u,v)^{-1}$  is the
quintuple
\begin{align*}
(0,0,bx-ay,2az-2cx+a^2y-bx^2,  2(c-ab)y-2b(z-xy)-ay^2+b^2x).
\end{align*}
It is easy to see that
\begin{enumerate}
\item $\mathcal{G}'=\{ (0,0,c,d,e) \mid c,d,e\in\mathbb{F}_q\}$.
\item $\gamma_3(\mathcal{G})=\{
(0,0,0,d,e) \mid d,e\in\mathbb{F}_q\} = \Z(\mathcal{G})$.
\end{enumerate}

Consider $(0,0,c,d,e)\in\mathcal{G}'\setminus \Z(\mathcal{G})$. Then $c\neq 0$.
If
$$[(0,0,c,d,e),(x,y,z,u,v)]={\mathbf 0},$$
then by the commutator formula above, $-2cx=2cy=0$. Since characteristic of
$\mathbb{F}_q$ is odd and $c\neq 0$, $x=y=0$. Noting that $\mathcal{G}'$ is
abelian, it follows that the centralizer of any element of
$\mathcal{G}'\setminus \Z(\mathcal{G})$
is $\mathcal{G}'$, which has index $q^2=p^{2m}$ in $\mathcal{G}$.

 Next fix
$g=(a,b,c,d,e)$ with $(a,b)\neq (0,0)$. Then $(x,y,z,u,v)$ centralizes $g$ if
and only if
\begin{align}
 & bx-ay=0,\label{eqse4-1}\\
& 2az-2cx+a^2y-bx^2=0,\label{eqse4-2}\\
 & 2(c-ab)y-2b(z-xy)-ay^2+b^2x=0.\label{eqse4-3}
\end{align}
Suppose $a\neq 0$. For arbitrary $x,u,v\in\mathbb{F}_q$, we see that $y$ is
uniquely determined from \eqref{eqse4-1} and then $z$ is uniquely determined
from \eqref{eqse4-2}. Further the values of $y,z$ obtained satisfy
\eqref{eqse4-3}. Hence the centralizer of $(a,b,c,d,e)$ has order $q^3=p^{3m}$,
and therefore has index $p^{2m}$ in $\mathcal{G}$.  Similarly, if $a=0$ and $b
\neq 0$, then it follows that the centralizer of $(a,b,c,d,e)$ in $\mathcal{G}$
has index $p^{2m}$. Hence $\mathcal{G}$
is of conjugate type $(1, p^{2m})$.

We  show that the group $\mathcal{G}$  has a nice
description in terms of a matrix group  over $\mathbb{F}_q$. Consider the
following collection of  unitriangular matrices over $\mathbb{F}_q$:
$$\mathcal{H}_m=
\begin{Bmatrix}
\begin{bmatrix}
1 &      &   &   &   \\
a & 1    &   &   &   \\
c & b    & 1 &   &   \\
d & ab-c & a & 1 &   \\
f & e    & c & b & 1 \\
\end{bmatrix}: a,b,c,d,e,f\in\mathbb{F}_q\end{Bmatrix}.
$$
It is easy to see that $\mathcal{H}_m$ is a {\it subgroup} of ${\rm U}_5(q)$.
Denote the general element  of $\mathcal{H}_m$ by $(a,b,c,d,e,f)$. An easy
computation shows that
$\Z(\mathcal{H}_m) = \{(0,0,0,0,0,f) \mid f \in \mathbb{F}_q\}$. Therefore, we
have a natural homomorphism $\mathcal{H}_m\rightarrow
\mathcal{H}_m/\Z(\mathcal{H}_m)$, in which we identify
$$
(a,b,c,d,e,f)\Z(\mathcal{H}_m) \longleftrightarrow (a,b,c,d,e).
$$
Then one can check that the product $(a,b,c,d,e)(x,y,z,u,v)$ in
$\mathcal{H}_m/\Z(\mathcal{H}_m)$ is the same  as $(a,b,c,d,e).(x,y,z,u,v)$ in
$\mathcal{G}$. Hence $\mathcal{H}_m/\Z(\mathcal{H}_m)$ is isomorphic to
$\mathcal{G}$, and therefore is of conjugate type $(1,p^{2m})$ and class $3$.

As a conclusion of the preceding discussion, we obtain

\begin{thm}\label{se4-thm1}
For any even integer $n \ge 1$ and an odd prime $p$, there exists a finite
$p$-group of nilpotency class $3$ and of conjugate type $(1, p^{n})$.
\end{thm}


\section{Proof of Main Theorem}
Let $G$ be a $p$-group of nilpotency class $3$ such  that $G/\Z(G)$ is
isomorphic to the group ${\rm U}_3(p^m)$. Our strategy for proving Main Theorem is to obtain
presentations of groups $G$ satisfying Hypothesis (A2) from a presentation of
${\rm U}_3(p^m)$; then we proceed to show that the groups, given by the
presentations obtained, belong to the same isoclinism family.

We start with finding some structure constants of ${\rm U}_3(p^m)$.
Let $\mathbb{F}_{p^m}$ denote the field of order $p^m$. Then
$\mathbb{F}_{p^m}=\mathbb{F}_p(\alpha)$, where $\alpha$ satisfies a monic
irreducible polynomial of degree $m$ over $\mathbb{F}_p$.
Consider the following matrices in ${\rm U}_3(p^m)$ for any integer $i \ge 1$:
$$X_i=
\begin{bmatrix}
1 &              & \\
0 &  1           & \\
0 & \alpha^{i-1} & 1
\end{bmatrix},
\hskip5mm
Y_i=
\begin{bmatrix}
1 &              & \\
\alpha^{i-1} &  1  & \\
0 & 0 & 1
\end{bmatrix},
\hskip5mm
H_i=
\begin{bmatrix}
1 &              & \\
0 &  1           & \\
\alpha^{i-1} & 0 & 1
\end{bmatrix}.
$$
Then it is easy to see that
$\{X_1,\ldots, X_m, Y_1,\ldots, Y_m\}$ is a minimal
generating set for
${\rm U}_3(p^m)$ and $\{H_1,\ldots, H_m\}$ is a minimal
generating set for the center (as well as the commutator subgroup) of ${\rm
U}_3(p^m)$. Further, these matrices satisfy the following
relations.
\begin{align}
& X_i^p=Y_i^p=H_i^p=1,\label{eqse5-1}\\
& [X_i,X_j]=[Y_i,Y_j]=1,\label{eqse5-2}\\
& [H_i,X_j]=[H_i,Y_j]=1,\label{eqse5-3}\\
& [X_i,Y_j]=H_{i+j-1} \mbox{ for all }i,j\geq 1.\label{eqse5-4}
\end{align}

Since $\mathbb{F}_p(\alpha)$ is a vector space over $\mathbb{F}_p$ with basis
$(1,\alpha,\ldots, \alpha^{m-1})$, for $\alpha^{i+j-2} \in\mathbb{F}_p(\alpha)$,
there exist unique
$\kappa_{i,j,1}, \ldots, \kappa_{i,j,m}\in\mathbb{F}_p$ such that
$$\alpha^{i+j-2} = \kappa_{i,j,1} + \kappa_{i,j,2}\,\alpha + \cdots +
\kappa_{i,j,m}\,\alpha^{m-1}.$$
Then, by \eqref{eqse5-4}, we have
$$H_{i+j-1} = H_1^{\kappa_{i,j,1}}H_2^{\kappa_{i,j,2}}\cdots
H_m^{\kappa_{i,j,m}}=
[X_1,Y_1 ^{\kappa_{i,j,1}} Y_2^{\kappa_{i,j,2}} \cdots Y_m^{\kappa_{i,j,m}}],$$
which in turn  implies that
\begin{align}\label{eqse5-5}
[X_i,Y_j]=[X_1,Y_1 ^{\kappa_{i,j,1}} Y_2^{\kappa_{i,j,2}} \cdots
Y_m^{\kappa_{i,j,m}}], \hskip5mm 1\leq i,j\leq m.
\end{align}

The constants $\kappa_{i,j,l}$ for $1\leq i,j,l\leq m$, which we call    {\it
the structure constants} of ${\rm U}_3(p^m)$, will be frequently used in the
remaining part of the paper. The generators $X_i,Y_i,H_i$ ($1\leq i\leq m$), the
constants $\kappa_{i,j,l}$, and the relations \eqref{eqse5-1} - \eqref{eqse5-5}
 give a presentation of the group ${\rm U}_3(p^m)$. From \eqref{eqse5-4}, it
follows  that $[X_i,Y_j]=[X_j,Y_i]$ for all $i,j \ge 1$, and so
\begin{equation}\label{eqse5-6}
\kappa_{i,j,l}=\kappa_{j,i,l}.
\end{equation}
Also, combining \eqref{eqse5-5} with the relation $[X_i,Y_j]=[X_j,Y_i]$,
we obtain
\begin{align}\label{eqse5-7}
[X_j,Y_i]=[X_1 ^{\kappa_{i,j,1}} X_2^{\kappa_{i,j,2}} \cdots
X_m^{\kappa_{i,j,m}},Y_1], \hskip5mm 1\leq i,j\leq m.
\end{align}

We now build up a presentation of a finite $p$-group $G$ such that $G/\Z(G)
\cong {\rm U}_3(p^m)$.

\begin{lemma}\label{lem10}
Fix a prime $p>2$ and an integer $m\ge 2$. Let $H$ be a $p$-group of
nilpotency class $3$ and of conjugate type $(1,p^{2m})$. Then there exists
$$\alpha_{i,j,l}, ~~ \beta_{i,j,l}, ~~ \gamma_{i,j,l}, ~~ \delta_{i,j,l}, ~~
\lambda_{i,j,l}, ~~ \mu_{i,j,l}, ~~ \epsilon_{i,l}, ~~ \nu_{i,l} \in
\mathbb{F}_p \hskip3mm (1\leq i, j \leq m, ~~ 1\leq l \le 2m),$$
such that $H$ is isoclinic to the group
$G$ admitting the following presentation:
\begin{align*}
 G=\Big{\langle}  &x_1,\ldots, x_m,y_1,\ldots, y_m, h_1, h_2,\ldots, h_m, z_1,
z_2, \ldots, z_{2m}~~ \Big{|} \\
 & x_i^p=\prod_{l=1}^{2m} z_l^{\epsilon_{i,l}}, ~~ y_i^p=\prod_{l=1}^{2m}
z_l^{\nu_{i,l}}, ~~ h_i^p=1 ~~ (1\leq i\leq m), ~~
z_i^p=1 ~~ (1\leq i\leq 2m),\tag{R0}\\
  & [z_k,z_r]=[z_k,x_i]=[z_k,y_i]=[z_k,h_i]=1 \hskip3mm (1\leq k,r,\leq 2m,
1\leq i\leq m), \tag{R1}\\
  & [h_i,h_j]=1 \hskip5mm (1\leq i,j\leq m),\tag{R2}\\
    & [h_i, x_j]=\prod_{l=1}^{2m} z_l^{\gamma_{i,j,l}} \hskip5mm (1\leq i,j\leq
m), \tag{R3}\\
  & [h_i, y_j]=\prod_{l=1}^{2m} z_l^{\delta_{i,j,l}} \hskip5mm (1\leq i,j\leq
m), \tag{R4}\\
  & [x_i,x_j]=\prod_{l=1}^{2m} z_l^{\alpha_{i,j,l}} \hskip5mm (1\leq i,j\leq m),
\tag{R5}\\
  & [y_i,y_j]=\prod_{l=1}^{2m} z_l^{\beta_{i,j,l}} \hskip5mm (1\leq i,j\leq m),
\tag{R6}\\
  & [x_i,y_j]=[x_1, y_1^{\kappa_{i,j,1}} y_2^{\kappa_{i,j,2}}\cdots
y_m^{\kappa_{i,j,m}}] \prod_{l=1}^{2m} z_l^{\lambda_{i,j,l}}   \hskip5mm (1\leq
i,j\leq m),\tag{R7}\\
  & [x_j,y_i]=[x_1^{\kappa_{i,j,1}} x_2^{\kappa_{i,j,2}} \cdots
x_m^{\kappa_{i,j,m}},y_1]
 \prod_{l=1}^{2m} z_l^{\mu_{i,j,l}}
  \hskip5mm (1\leq i,j\leq m), \tag{R8}\\
  & [x_1, y_i] = h_i, ~~[h_1, x_i] = z_i, ~~[h_1, y_i] = z_{m+i} \hskip3mm
(1\leq i\leq m) \tag{R9}
   \Big{\rangle},
  \end{align*}
  where   $\kappa_{i,j,l}$,  $1 \le i, j, l \le m$, are the
structure constants of ${\rm U}_3(p^m)$.
\end{lemma}
\begin{proof}
By Propositions \ref{prop1} and \ref{prop2}, $H$ is isoclinc to a group $G$
satisfying hypothesis (A1). Then  by Lemma \ref{lem4}, Theorem \ref{lem6} and
Lemma \ref{lem9},
$$[G:G']=p^{2m}, ~~ [G':\Z(G)]=p^m, ~~ |\Z(G)|=p^{2m}.$$
The desired presentation of $G$ is obtained
from the presentation of ${\rm U}_3(p^m)$ $( \cong G/\Z(G)$) described just
before the lemma. Note that, by Theorem \ref{ishikawa2}, $G'$ is elementary
abelian.

Since $|\Z(G)|=p^{2m}$ and $\Z(G)$ is of exponent $p$, $\Z(G)$ is minimally
generated by $2m$ elements
$z_1,z_2,\ldots,
z_{2m}$ (say).
Let
$\varphi:G\rightarrow {\rm U}_3(p^m)$ denote a surjective homomorphism with
$\ker\varphi=\Z(G)$. Let $X_i$'s, $Y_i$'s and $H_i$'s be the generators of ${\rm
U}_3(p^m)$ considered in the discussion preceding the lemma.

Choose $ x_1,\ldots, x_m, y_1,\ldots, y_m, h_1,\ldots, h_m$
in $G$ such that
$$\varphi(x_i)=X_i, \hskip3mm \varphi(y_i)=Y_i\hskip3mm \mbox{ and } \hskip3mm
\varphi(h_i)=H_i\hskip3mm (1\leq i\leq m).$$
It follows that the set
$$\{x_1,\ldots, x_m,y_1,\ldots, y_m, h_1, h_2,\ldots, h_m, z_1,
z_2, \ldots, z_{2m}\}$$
generates $G$, and
$$\{h_1, h_2,\ldots, h_m, z_1, z_2, \ldots, z_{2m}\}$$
generates $G'$.

Since $\varphi(x_i^p)=1$, $1\le i\le m$, there exist some
$\epsilon_{i,l}\in\mathbb{F}_p$, $1\le l\le 2m$, such that $x_i^p=\prod_l
z_l^{\epsilon_{i,l}}$. Similarly there
exist some $\nu_{i,l}\in\mathbb{F}_p$,
$1\le l\le 2m$, such that $y_i^p=\prod_l z_l^{\nu_{i,l}}$, $1\le i\le m$.
Since $G'$ is elementary abelian, we get
$h_i^p=1$, $[h_i,h_j]=1$, $1\le i,j\le m$  and $z_i^p=1$, $1\le i\le 2m$. This
gives all the relations in (R0) and (R2). Since $z_i\in \Z(G)$, the relations
(R1) are the obvious commutator
relations of $z_i$'s with the other generators of $G$. The relations
\eqref{eqse5-2},  \eqref{eqse5-3}, \eqref{eqse5-5} and
\eqref{eqse5-7} of ${\rm U}_3(p^m)$ give the
relations (R3)-(R8) {\it for some} $\alpha$'s, $\beta$'s, $\gamma$'s,
$\delta$'s, $\lambda$'s and $\mu$'s in $\mathbb{F}_p$.

It remains to obtain relations (R9). Since, for $1 \le i \le m$,  $[X_1, Y_i]
= H_i$, by the choice of $x_1$, $y_i$
and $h_i$, we have
$$[x_1, y_i] \equiv h_i \pmod{\Z(G)}.$$ Thus $[x_1, y_i] = h_iw_i$ for some $w_i
\in \Z(G)$. Replacing $h_i$ by $h_iw_i$, which do not violate any of the
preceding relations, we can assume, without loss of generality, that
\begin{equation}\label{good-eq}
[x_1, y_i] = h_i, \hskip3mm 1 \le i \le m.
\end{equation}
Finally, we have $h_1\in G'\setminus \Z(G)$. Further,
$G$ is of conjugate type $(1,p^{2m})$, and $[G:G']=p^{2m}$ with
$G'$ abelian. Therefore $C_G(h_1)=G'$.
Since $x_1,\ldots, x_m, y_1,\ldots, y_m$ are independent modulo $C_G(h_1)=G'$,
it follows that
the $2m$ commutators
$$[h_1,x_1], \ldots, [h_1,x_m], [h_1,y_1], \ldots,
[h_1,y_m]$$
are independent, and they belong to $\gamma_3(G)=\Z(G)$, which is elementary
abelian of order $p^{2m}$.
Thus, without
loss of generality, we can take $[h_1,x_i]=z_i$  and $[h_1,y_i]=z_{m+i}$  for $1
\le i \le m$. This, along with \eqref{good-eq}, gives relations  (R9).
\end{proof}

\begin{thm}\label{se5-thm1}
Given an even integer $n=2m \ge 2$ and an odd prime $p$, there is a unique
finite $p$-group satisfying Hypothesis (A2), upto isoclinism. Moreover, such a group  is isoclinic to the group $\mathcal{H}_m/\Z(\mathcal{H}_m)$.
\end{thm}

Assuming the preceding theorem, we are now ready to present
\vspace{3mm}

\noindent {\it Proof of Main Theorem.}
For an integer $n\geq 1$ and a prime $p$, let $G$ be a finite $p$-group of
nilpotency class $3$ and of conjugate type $(1,p^n)$. In view of Propositions
\ref{prop1} and \ref{prop2}, we can assume that $\Z(G)\le G'$.
Then, by Theorem \ref{lem6},
$n$ is even. Conversely, if $n\geq 1$ is an even integer and $p$ is an odd
prime, then that there exists a finite $p$-group of nilpotency class $3$ and of
conjugate type $(1, p^n)$, follows from Theorem \ref{se4-thm1}. This proves the
first assertion of Main Theorem. The second one follows from Theorem
\ref{se5-thm1}. \hfill $\Box$

\vspace{3mm}

The rest of this section is devoted to the proof of Theorem \ref{se5-thm1}.
Throughout, for a prime $p$, and integer $m\ge 1$,   {\it $G$ (or more
precisely, $G(\alpha,\beta,\gamma,\delta,\lambda,\mu,\epsilon,\nu)$) always
denotes a
group of conjugate type $(1,p^{2m})$, admitting  the presentation as in
Lemma \ref{lem10}},  where $\alpha$'s,
$\beta$'s, $\gamma$'s, $\delta$'s, $\lambda$'s, $\mu$'s, $\epsilon$'s, and
$\nu$'s belong to $\mathbb{F}_p$. For $m=1$,
the Theorem \ref{se5-thm1} has
been proved by Ishikawa in \cite[Theorem 4.2]{Ishikawa1999}. Thus, assume
that $m\geq 2$.  Further, let the homomorphism $\varphi:G\rightarrow {\rm
U}_3(p^m)$,
$$x_i\mapsto X_i, \hskip5mm  y_i\mapsto Y_i, \mbox{ and  }h_i\mapsto H_i$$
be  as in the proof of Lemma \ref{lem10}.

\vspace{3mm}

First, we proceed to show that there is a unique
choice for $\alpha$'s,
$\beta$'s, $\gamma$'s, $\delta$'s, $\lambda$'s and $\mu$'s for which the group
$G$ considered in Lemma \ref{lem10} is of conjugate type $(1, p^{2m})$.
We start with simplifying the relations of  $G$ without any loss of generality.

\begin{lemma}\label{lem12}
We can choose generators $x_j$'s and $y_j$'s of $G$ such that
 $$[x_1,x_j]=[y_1,y_j]=1 \hskip5mm (j=2,\ldots, m).$$
\end{lemma}
\begin{proof}
Fix $j$ with $2\le j\le m$. Since $x_1,x_j\notin G'$ and
$[x_1,x_j]\in \Z(G)$, by Corollary
\ref{se3-cor1},
there exists $h_j\in G'$ such that $[x_1,x_jh_j]=1$. Thus replacing $x_j$
by $x_jh_j$, if necessary, we can assume that $[x_1,x_j]=1$. Similarly, we can
choose $y_j$'s such that $[y_1,y_j]=1$.
\end{proof}

\begin{thm}\label{lem14}
In the relations (R3) and (R4),  for $1\leq i,j\leq m$, we have
$$
 [h_i,x_j]=[h_j,x_i]=z_1^{\kappa_{i,j,1}} z_2^{\kappa_{i,j,2}}\cdots
z_m^{\kappa_{i,j,m}} $$
and
$$ [h_i,y_j]=[h_j,y_i]=z_{m+1}^{\kappa_{i,j,1}}
z_{m+2}^{\kappa_{i,j,2}} \cdots
z_{2m}^{\kappa_{i,j,m}}.$$

In particular, $\gamma$'s and $\delta$'s (in the relations (R3) and
(R4)) are uniquely determined by the structure constants $\kappa_{i,j,l}$
of ${\rm
U}_3(p^m)$.
\end{thm}
\begin{proof}
 Since
$H_i=[X_1,Y_i]$   in ${\rm U}_3(p^m)$ for $1 \le i \le m$, we get $h_i\equiv
[x_1,y_i]\pmod{\Z(G)}$ in $G$. Therefore
$$[h_i,x_j]=[[x_1,y_i],x_j].$$

Since $[x_1,x_j]\in \Z(G)$, by \eqref{eqse3-5a}, we get
$$[h_i,x_j]=[[x_1,y_i], x_j]=[[x_j,y_i],x_1].$$
From the relation (R8), $[x_j,y_i]\equiv [x_1^{\kappa_{i,j,1}}
x_2^{\kappa_{i,j,2}}
\cdots x_m^{\kappa_{i,j,m}},y_1]\pmod{\Z(G)}$.  Then, again by \eqref{eqse3-5a},
we get
\begin{align*}
[h_i,x_j]=[x_j,y_i,x_1] &=[x_1^{\kappa_{i,j,1}} x_2^{\kappa_{i,j,2}}
\cdots x_m^{\kappa_{i,j,m}},y_1,x_1]\\
&=[x_1,y_1,x_1^{\kappa_{i,j,1}} x_2^{\kappa_{i,j,2}}
\cdots x_m^{\kappa_{i,j,m}}]  \\
&=[h_1,x_1^{\kappa_{i,j,1}} x_2^{\kappa_{i,j,2}}
\cdots x_m^{\kappa_{i,jm}}]\\
&=z_1^{\kappa_{i,j,1}} z_2^{\kappa_{i,j,2}}\cdots
z_m^{\kappa_{i,j,m}}.
\end{align*}
Since the structure constants $\kappa_{i,j,l}$ are symmetric in $i,j$ (see
\eqref{eqse5-6}),
$[h_i,x_j]=[h_j,x_i]$. This proves the first assertion of the Lemma, and
the second one goes on the same lines.
\end{proof}

As an immediate consequence of the preceding result, we have

\begin{cor}\label{lem15}
Consider the elements $x=x_1^{a_1}\cdots x_m^{a_m}$, $x'=x_1^{b_1} \cdots
x_m^{b_m}$,
$y=y_1^{c_1}\cdots y_m^{c_m}$ and $y'=y_1^{d_1}\cdots y_m^{d_m}$ in $G$. If
$$[x,y,x']=\prod_{l=1}^{2m} z_l^{r_l}~~ \mbox{ and } ~~[y,x,y']=\prod_{l=1}^{2m}
z_l^{s_l},$$
for some $r_l,s_l\in \mathbb{F}_p$, then the values of $r_l$ and $s_l$ are
uniquely determined by the structure constants of ${\rm U}_3(p^m)$ and,
respectively, by $a_i$'s, $b_i$'s, $c_i$'s  and  $a_i$'s, $c_i$'s, $d_i$'s.
\end{cor}

Up to now, we have shown the uniqueness of $\gamma$'s and $\delta$'s (Theorem
\ref{lem14}). Next we  show the uniqueness of $\lambda$'s and $\mu$'s. First we
prove some lemmas.

\begin{lemma}\label{lem16}
For $1 \le i, j \le m$, the following hold:
\begin{enumerate}
\item $[x_i,x_j] \in \gen{z_1,\ldots, z_m}$.
\item $[y_i,y_j] \in \gen{z_{m+1},\ldots, z_{2m}}$.
\end{enumerate}
\end{lemma}
\begin{proof}
Fix $i,j$ with $1\le i,j\leq m$. Since $[x_i,x_j]\in \Z(G)$, by
Corollary  \ref{se3-cor1}, there exists $h\in G'$ (depending on $x_i,x_j$) such
that
$[x_i,x_jh]=1$. Then
$$[x_i,x_j]=[x_i,h]^{-1}=[h,x_i].$$
Since $G'=\langle h_1,h_2,\ldots,h_m,\Z(G)\rangle$, by
Theorem \ref{lem14}, $[h,x_i]\in\langle z_1,\ldots,
z_m\rangle$,
 proving the first assertion.
The second assertion follows on the same lines.
\end{proof}

Before proceeding further, we recall the following commutator identities in a
finite $p$-group $G$ of nilpotency class $3$ with $p$ odd, which will be used in
computations without any reference. Note that, in  this case  $G'$ is abelian, so
the ordering of commutators is immaterial. For $a,b,c\in G$,
\begin{enumerate}
 \item $[ab,c]=[a,c][b,c][a,c,b]$ and $[a,bc]=[a,b][a,c][a,b,c]$;
 \item $[a^i,b^j,c^k]=[a,b,c]^{ijk}$ (since $G$ is of class $3$);
 \item $[a^s,b]=[a,b]^s [a,b,a]^{\binom{s}{2}}$ and $[a,b^s]=[a,b]^s
[a,b,b]^{\binom{s}{2}}$ \hskip3mm
 ($s\in\mathbb{F}_p$),
 \end{enumerate}
 where $\binom{s}{2}=\frac{s(s-1)}{2}$ in $\mathbb{F}_p$, $p>2$.

\begin{lemma}\label{lem17}
For $x=x_1^{a_1}\cdots x_m^{a_m}$ and $y=y_1^{a_1}\cdots
y_m^{a_m}$ in $G$, the following hold:
\begin{enumerate}
\item $[x_1,y]\equiv
[x,y_1]\pmod{\Z(G)}$.
\item If
$[x_1,y]=[x,y_1]\prod_{l=1}^{2m} z_l^{c_l},$
then $c_l$'s are uniquely determined by $a_i$'s and  the structure
constants of ${\rm U}_3(p^m)$.
\end{enumerate}
\end{lemma}
\begin{proof}
 If $a_i=0$ for all $i$, then $x=y=1$, and so $c_i=0$ for
all $i$, there is nothing to prove. Thus, assume that not all $a_i$'s are $0$.
In ${\rm U}_3(p^m)$, consider $X=X_1^{a_1}\cdots
X_m^{a_m}$ and $Y=Y_1^{a_1}Y_2^{a_2}\cdots Y_m^{a_m}$. Since
$[X_i,Y_j]=[X_j,Y_i]$ for all $i,j$ (by \eqref{eqse5-4}), we have
$[X_1,Y]=[X,Y_1]$; hence
 $[x_1,y]\equiv [x,y_1]\pmod{\Z(G)}$ in  $G$. This proves assertion (1).

We prove assertion (2) in two steps, as the arguments of  Step 1 will
be used further.

\vskip3mm
\noindent{\bf Step 1.}  {\it For $m+1 \le l \le 2m$, $c_l$ is uniquely
determined by the  structure constants of ${\rm U}_3(p^m)$.}
\vskip3mm

With $X,Y$ as in the proof of assertion (1), we have $[X_1,Y]=[X,Y_1]$, which
implies that
 $[X_1Y_1^r, XY^r]=1$ for any $r \in\mathbb{F}_p$; hence, in $G$,  we get
$$[x_1y_1^r, xy^r]\equiv 1\pmod{\Z(G)}.$$
By Corollary \ref{se3-cor1}, there exists an element $h(r)$, depending on $r$,
in $G'$ such that
$$[x_1y_1^r, xy^rh(r)]=1.$$

Since $[x_1,x]=[y_1,y]=1$ (see Lemma \ref{lem12}), we have
\begin{align*}
1 = [x_1y_1^r, xy^rh(r)] = &[x_1,y]^r[x_1,y,y]^{\binom{r}{2}}
[x_1,y,y_1]^{r^2}[x_1,h(r)]\\
 & [y_1,x]^r [y_1,x,y_1]^{\binom{r}{2}}[y_1, x,y]^{r^2}[y_1,h(r)]^r.
 \end{align*}
Consequently, by the given hypothesis, we get
\begin{align}
 (\prod_{l=1}^{2m} z_l^{c_l})^{-r}= ([x_1,y][y_1,x])^{-r}
= &[x_1,y,y]^{\binom{r}{2}} [x_1,y,y_1]^{r^2}[x_1,h(r)]\label{eqse5-9}\\
 & [y_1,x,y_1]^{\binom{r}{2}}[y_1, x,y]^{r^2}[y_1,h(r)]^r .\nonumber
\end{align}
By Theorem \ref{lem14}, $[G',y], [G',y_1]\leq \langle z_{m+1},\ldots,
z_{2m}\rangle$ and
$[x_1,h(r)]\in \langle z_1,\ldots, z_m\rangle$; hence  all the commutators on
the right
side of \eqref{eqse5-9}, except $[x_1,h(r)]$, belong to $\langle z_{m+1},\ldots,
z_{2m}\rangle$.
Thus
\begin{align}\label{eqse5-10}
(\prod_{l=1}^{m} z_l^{-c_l} )^{r} &=[x_1,h(r)]
\end{align}
and
\begin{align}\label{eqse5-11}
(\prod_{l=m+1}^{2m} z_l^{-c_l})^{r} &=[x_1,y,y]^{\binom{r}{2}} [x_1,y,y_1]^{r^2}
[y_1,x,y_1]^{\binom{r}{2}}[y_1, x,y]^{r^2}[y_1,h(r)]^{r}.
\end{align}

Since, by Theorem \ref{lem14},
$z_i = [h_1,x_i] = [h_i ,x_1]$ for $ 1\le i \le m$, we have
$$\prod_{l=1}^{m} z_l^{-c_l} = \prod_{l=1}^{m} [x_1, h_l]^{c_l} = [x_1,
h_1^{c_1} \cdots  h_m^{c_m}] = [x_1, \tilde{h}],$$
where $\tilde{h} = h_1^{c_1} \cdots  h_m^{c_m}$, which is independent of $r$.
Then from \eqref{eqse5-10} we get
$[x_1,h(r)]=[x_1,\tilde{h}]^r$, which implies that $\tilde{h}^rh(r)^{-1}\in
C_G(x_1)\cap G'=\Z(G)$.
Hence
 $$h(r)\equiv \tilde{h}^r\pmod{\Z(G)}.$$

Using this in \eqref{eqse5-11}, we get
$$(\prod_{l=m+1}^{2m} z_l^{-c_l})^{r} =[x_1,y,y]^{\binom{r}{2}}
[x_1,y,y_1]^{r^2}
[y_1,x,y_1]^{\binom{r}{2}}[y_1, x,y]^{r^2}[y_1,\tilde{h}]^{r^2}.$$
Since this equation holds for all $r\in\mathbb{F}_p$, for $r=1$ and $r=-1$,
we, respectively, get
\begin{align*}
\prod_{l=m+1}^{2m} z_l^{-c_l}&=[x_1,y,y_1] [y_1, x,y] [y_1,\tilde{h}]
\end{align*}
and
\begin{align*}
\prod_{l=m+1}^{2m} z_l^{c_l}& =[x_1,y,y]  [x_1,y,y_1] [y_1,x,y_1] [y_1, x,y]
[y_1,\tilde{h}].
\end{align*}
From these  equations, we obtain
$$\prod_{l=m+1}^{2m} z_l^{2c_l} =[x_1,y,y][y_1,x,y_1].$$

By Corollary \ref{lem15}, the right hand side of the preceding equation is
uniquely determined by $a_i$'s and  the structure
constants of ${\rm U}_3(p^m)$,  so are $c_l$ for $ m+1 \le l \le 2m$.

\vskip3mm
\noindent{\bf Step 2.}  {\it  For $1 \le l \le m$, $c_l$ is uniquely determined
by the  structure constants of ${\rm U}_3(p^m)$.}
\vskip3mm

With $X,Y$ as in the proof of assertion (1), we have $[Y, X_1]=[Y_1, X]$, which
implies that
 $[Y_1X_1^r, YX^r]=1$ for any $r \in\mathbb{F}_p$; hence, in $G$,  we get
$$[y_1x_1^r, yx^r]\equiv 1\pmod{\Z(G)}.$$

By Corollary \ref{se3-cor1}, there exists an element $k(r)$, depending on $r$,
in $G'$ such that
$[y_1x_1^r, yx^rk(r)] = 1$. With appropriate modifications in Step (1), it
follows that $c_l$ for $1 \le l  \le m$ are uniquely determined by  $a_i$'s and
the structure constants of ${\rm U}_3(p^m)$.

Finally by Step 1 and Step 2, $c_l$ for $1 \le l  \le 2m$ are
uniquely determined by  $a_i$'s and the structure constants of ${\rm U}_3(p^m)$.
\end{proof}

\begin{lemma}\label{lem18}
In the relations (R7), namely, for $1 \le i, j \le m$,
\begin{align}\label{eqse5-12}
& [x_i,y_j]=[x_1, y_1^{\kappa_{i,j,1}} y_2^{\kappa_{i,j,2}}\cdots
y_m^{\kappa_{i,j,m}}]  \prod_{l=1}^{2m} z_l^{\lambda_{i,j,l}},
  \end{align}
  $\lambda_{i,j,l}$,  $m+1 \le l \le 2m$, are uniquely
determined by the structure constants of ${\rm
U}_3(p^m)$.
\end{lemma}
\begin{proof}
For simplicity, we fix $i,j\leq m$ and write
$y_1^{\kappa_{i,j,1}}y_2^{\kappa_{i,j,2}}\cdots
y_m^{\kappa_{i,j,m}} = y.$
Since $[x_i,y_j]\equiv [x_1,y]\pmod{\Z(G)}$, we have
$[x_1y_j^r,x_iy^r]\equiv 1\pmod{\Z(G)}$ for any
$r\in\mathbb{F}_p$.

That $\lambda_{i,j,l}$,
$ m+1 \le l \le 2m$, are uniquely determined by the structure constants of ${\rm
U}_3(p^m)$, follows on the lines (without any extra work) of  Step 1 of Lemma
\ref{lem17}.
\end{proof}

By the symmetry of  $x_i$'s and $y_i$'s,  the following lemma is dual of the
preceding one.

\begin{lemma}\label{lem19}
In the relations (R8), namely, for $1 \le i, j \le m$,
\begin{align}\label{eqse5-13}
 & [x_j,y_i]=[x_1^{\kappa_{i,j,1}} x_2^{\kappa_{i,j,2}}\cdots
x_m^{\kappa_{i,j,m}},y_1]  \prod_{l=1}^{2m} z_l^{\mu_{i,j,l}},
  \end{align}
$\mu_{i,j,l}$,  $1 \le l \le m$, are uniquely determined by the structure
constants of ${\rm U}_3(p^m)$.
\end{lemma}

We are now ready to prove the uniqueness of $\lambda$'s and $\mu$'s.
\begin{thm}\label{se5-thm4}
In the relations (R7) and (R8),  $\lambda_{i,j,l}$ and
$\mu_{i,j,l}$, $1 \le i, j \le m$, $1 \le l \le 2m$, are
uniquely determined by the structure constants of ${\rm U}_3(p^m)$.
\end{thm}
\begin{proof}
The proof involves a careful application of  Lemmas \ref{lem17}, \ref{lem18} and
\ref{lem19}.

 Interchanging  $i$ and $j$ in \eqref{eqse5-13}, and
noting that  $\kappa_{i,j,l} = \kappa_{j,i,l}$, we get
\begin{align}\label{eqse5-14}
[x_i,y_j]=[x_1^{\kappa_{i,j,1}} x_2^{\kappa_{i,j,2}}\cdots
x_m^{\kappa_{i,j,m}},y_1]  \prod_{l=1}^{2m} z_l^{\mu_{j,i,l}}.
\end{align}
By Lemma \ref{lem19}, $\mu_{j,i,l}$,  $1\leq i,j,l\leq m$, are
uniquely determined by the structure constants of ${\rm
U}_3(p^m)$.

By \eqref{eqse5-12} and \eqref{eqse5-14}, we get
$$[x_1,y]=[x,y_1]\prod_{l=1}^{2m} z_l^{\mu_{j,i,l}-\lambda_{i,j,l}},$$
where $x=x_1^{\kappa_{i,j,1}}\cdots x_m^{\kappa_{i,j,m}}$ and
$y=y_1^{\kappa_{i,j,1}}\cdots y_m^{\kappa_{i,j,m}}$.
By Lemma \ref{lem17},
$\mu_{j,i,l}-\lambda_{i,j,l}$,   $ 1\leq i,j,\leq m$, $1 \le l \le 2m$, are
uniquely determined by
the structure constants of ${\rm U}_3(p^m)$.

Hence $\lambda_{i,j,l}$, $ 1\leq i,j,\leq m$, $1 \le l \le m$, are uniquely
determined by the structure constants of ${\rm U}_3(p^m)$. This, together with
Lemma \ref{lem18}, completes the uniqueness of $\lambda$'s.

Similarly,  $\mu_{i,j,l}$, $1\leq i,j\leq m$,  $1 \le l \le 2m$, are uniquely
determined by
the structure constants of ${\rm U}_3(p^m)$.
\end{proof}

Finally we prove the  uniqueness of $\alpha$'s and $\beta$'s in the relations
(R5) and (R6). Although the  proofs are  almost similar to the proofs of the
previous cases,  these can not go verbatim. Thus, we give complete proof of
uniqueness of $\alpha$'s and $\beta$'s.

We need the following preliminary lemma.

\begin{lemma}\label{lem20}
For any $i$ with $1\leq i\leq m$ and  $c_j,c_j', d_j, d_j' \in \mathbb{F}_p$ for
$1 \le j \le m$, assume that the commutator
equations
$$[h_1^{c_1} \cdots h_m^{c_m}, x_i]=z_1^{d_1} \cdots z_m^{d_m}$$
and
$$[h_1^{c_1'} \cdots h_m^{c_m'}, y_i]=z_{m+1}^{d_1'}  \cdots z_{2m}^{d_m'}$$
hold in the group $G$.
Then there exists an $m\times m$ invertible matrix $A$, whose entries are the
structure
constants $\kappa_{i,j,l}$ of ${\rm U}_3(p^m)$, such that
$$\begin{bmatrix} c_1 &  \cdots & c_m\end{bmatrix}A= \begin{bmatrix}d_1
&  \cdots & d_m\end{bmatrix}
$$
and
$$\begin{bmatrix} c_1' &  \cdots & c_m'\end{bmatrix}A = \begin{bmatrix}d_1'
&  \cdots & d_m'\end{bmatrix}
$$
\end{lemma}
\begin{proof}
Since $h_1,h_2,\ldots, h_m$ are independent modulo $C_G(x_i)$, it follows that
$[h_1,x_i]$, $[h_2,x_i]$, $\ldots$, $[h_m,x_i]$ are independent, and lie
in
$\langle z_1,z_2,\ldots, z_m\rangle$ (by Theorem \ref{lem14}).
In other words,
$\{[h_1,x_i], [h_2,x_i],\ldots, [h_m,x_i]\} $ is also a basis for the vector
space $\langle z_1,z_2,\ldots, z_m\rangle$ over $\mathbb{F}_p$.
From Theorem \ref{lem14}, $[h_j,x_i]=z_1^{\kappa_{j,i,1}}
z_2^{\kappa_{j,i,2}}\cdots
z_m^{\kappa_{j,i,m}}$. The desired matrix $A$ is the matrix of change of basis
from $\{z_1,\ldots, z_m\}$ to $\{[h_1,x_i],$ $ [h_2,x_i],\ldots, [h_m,x_i]\}$,
which is given by
$$A = (a_{r,s}), $$
where $a_{r,s} =  \kappa_{r,i,s}$ for $1 \le r, s \le m$.

Changing $x_i$ by $y_i$ and $z_i$ by $z_{m+i}$, similarly, we get the second
assertion.
\end{proof}

\begin{thm}\label{se5-thm5}
In the relations (R5) and (R6),   $\alpha_{i,j,l}$  and $\beta_{i,j,l}$, $1 \le
i, j \le m$, $1 \le l \le 2m$,  are uniquely determined  by
the structure constants of ${\rm U}_3(p^m)$.
\end{thm}
\begin{proof}
Fix $i,j$ with $1\leq i,j\leq m$. Consider the relations (R7):
\begin{align*}
[x_i,y_j]=[x_1, y_1^{\kappa_{i,j,1}} y_2^{\kappa_{i,j,2}}\cdots
y_m^{\kappa_{i,j,m}}]  \prod_{l=1}^{2m} z_l^{\lambda_{i,j,l}}.
\end{align*}
Since $\kappa_{i,j,l}=\kappa_{j,i,l}$, interchanging $i$ and $j$,   we get
\begin{align*}
[x_j,y_i]=[x_1, y_1^{\kappa_{i,j,1}} y_2^{\kappa_{i,j,2}}\cdots
y_m^{\kappa_{i,j,m}}] \prod_{l=1}^{2m} z_l^{\lambda_{j,i,l}}.
\end{align*}

From the preceding  two equations, we obtain
\begin{equation}\label{eqse5-15}
[x_i,y_j]=[x_j,y_i]\prod_{l=1}^{2m}
z_l^{\lambda_{i,j,l}-\lambda_{j,i,l}};
\end{equation}
hence $[x_i,y_j]\equiv [x_j,y_i]\pmod{\Z(G)}$. Since
$[x_i,x_j]\equiv [y_i,y_j]\equiv 1\pmod{\Z(G)}$, it follows that $[x_iy_i^r,
x_jy_j^r]\equiv
1$ $\pmod{\Z(G)}$ for all $r \in {\mathbb{F}_p}$. By Corollary \eqref{se3-cor1},
there
exists an element $k(r)$, depending on $r$, in $G'$ such that
$[x_iy_i^r, x_jy_j^rk(r)]=1$. Then
\begin{align*}
1  =  [x_iy_i^r, x_jy_j^rk(r)]
= &[x_i,x_j] [x_i,y_j]^r [x_i,y_j,y_j]^{\binom{r}{2}}
[x_i,y_j,y_i]^{r^2}[x_i,k(r)]\\
 & [y_i,x_j]^r[ y_i,x_j,y_i]^{\binom{r}{2}} [y_i,x_j,y_j]^{r^2} [y_i,y_j]^{r^2}
[y_i,k(r)]^r
\end{align*}

 Using \eqref{eqse5-15}, we get
\begin{align*}
\Big{(} \prod_{l=1}^{2m}
z_l^{\lambda_{i,j,l}-\lambda_{j,i,l}}\Big{)}^{-r}=
&\Big{(} [x_i, x_j][x_i,k(r)]\Big{)}
\Big{(} [x_i,y_j,y_j]^{\binom{r}{2}} [x_i,y_j,y_i]^{r^2} \\
& [ y_i,x_j,y_i]^{\binom{r}{2}} [y_i,x_j,y_j]^{r^2}[y_i, y_j]^{r^2}
[y_i,k(r)]^r \Big{)}.
\end{align*}

Using Theorem \ref{lem14} and Lemma \ref{lem16}, it follows that
$[x_i, x_j] [x_i,k(r)] \in \langle z_1,\ldots, z_m\rangle$
and the rest of the terms in the right side of the preceding equation belong to
$\langle z_{m+1}$, $\ldots$, $z_{2m}$ $\rangle$.
Thus
\begin{align}\label{eqse5-16}
\Big{(} \prod_{l=1}^{m}
z_l^{\lambda_{j,i,l}-\lambda_{i,j,l}}\Big{)}^{r}=[x_i,x_j]
[x_i,k(r)]
\end{align}
and
\begin{align}\label{eqse5-17}
\Big{(} \prod_{l=m+1}^{2m}
z_l^{\lambda_{j,i,l}-\lambda_{i,j,l}}\Big{)}^{r}= &[x_i,y_j,y_j]^{\binom{r}{2}}
[x_i,y_j,y_i]^{r^2}  [ y_i,x_j,y_i]^{\binom{r}{2}}\\ & [y_i,x_j,y_j]^{r^2}
[y_i,y_j]^{r^2}
 [y_i,k(r)]^r. \nonumber
\end{align}

Since $[h_1,x_i], [h_2,x_i],\ldots, [h_m,x_i]$ are independent and belong to
$\langle z_1,z_2,\ldots,z_m\rangle$ (see Theorem \ref{lem14}), we get
\begin{align*}
\langle z_1,z_2,\ldots, z_m\rangle =\langle [h_1,x_i], [h_2,x_i],\ldots.
[h_m,x_i]\rangle,
\end{align*}
Hence there exists $\tilde{k}\in G'$, independent of $r$,
such that $\prod_{l=1}^{m}
z_l^{-\lambda_{i,j,l}+\lambda_{j,i,l}}=[x_i,\tilde{k}]$. Similarly
there exists $\hat{k}\in G'$,
independent of $r$, such that
\begin{equation}\label{eqse5-18}
[x_i,x_j]=[x_i,\hat{k}^{-1}].
\end{equation}
Therefore from \eqref{eqse5-16}, we
get $[x_i,\tilde{k}]^r=[x_i,\hat{k}^{-1}][x_i,k(r)]$, which implies that
$\tilde{k}^r\hat{k}k(r)^{-1}\in C_G(x_i)\cap G'=\Z(G)$. Hence
$$k(r)\equiv \tilde{k}^r\hat{k}\pmod{\Z(G)}.$$

Using this in \eqref{eqse5-17}, we get
\begin{align*}
 \prod_{l=m+1}^{2m} z_l^{r(\lambda_{j,i,l}-\lambda_{i,j,l})}=
[x_i,y_j,y_j]^{\binom{r}{2}} [x_i,y_j,y_i]^{r^2}
[ y_i,x_j,y_i]^{\binom{r}{2}}
[y_i,x_j,y_j]^{r^2}[y_i,\tilde{k}]^{r^2}[y_i,y_j]^{r^2}[y_i,\hat{k}]^r.
\end{align*}
Writing this equation for $r=1$ and $r=-1$, we get
\begin{align*}
\prod_{l=m+1}^{2m} z_l^{\lambda_{j,i,l}-\lambda_{i,j,l}}=&
 [x_i,y_j,y_i][y_i,x_j,y_j][y_i,\tilde{k}][y_i,y_j][y_i,\hat{k}]
\end{align*}
and
\begin{align*}
\prod_{l=m+1}^{2m} z_l^{\lambda_{i,j,l}-\lambda_{j,i,l}}=[x_i,y_j,y_j]
[x_i,y_j,y_i] [
y_i,x_j,y_i][y_i,x_j,y_j][y_i,\tilde{k}][y_i,y_j][y_i,\hat{k}]^{-1}.
\end{align*}

Preceding two equations give
$$\prod_{l=m+1}^{2m} z_l^{2(\lambda_{j,i,l}-\lambda_{i,j,l})} =
[x_i,y_j,y_j]^{-1} [y_i,x_j,y_i]^{-1} [y_i,\hat{k}]^2.$$
Rearranging the terms, we get
$$\prod_{l=m+1}^{2m} z_l^{(\lambda_{j,i,l}-\lambda_{i,j,l})}
[x_i,y_j,y_j]^{1/2} [x_i,y_j,y_i]^{1/2} =[y_i,\hat{k}].$$
Note that the left side of the preceding equation is of the form
$z_{m+1}^{d_1} \cdots z_{2m}^{d_m}$. By Corollary \ref{lem15} and
Theorem \ref{se5-thm4}, $d_i$'s are uniquely determined  by the structure
constants of
${\rm U}_3(p^m)$. Let $\hat{k} =
h_1^{c_1} \cdots h_m^{c_m}\pmod{\Z(G)}.$ Thus, we have the commutator equation
$$ z_{m+1}^{d_1} \cdots z_{2m}^{d_m}=[y_i,h_1^{c_1} \cdots h_m^{c_m}].$$
By Lemma \ref{lem20}, it follows that $c_i$'s are uniquely determined by the
structure constants of ${\rm U}_3(p^m)$.

By  Lemma \ref{lem16}, $\alpha_{i,j,l} = 0$ for $l
> m$. Then by \eqref{eqse5-18}, we have
$$\prod_{l=1}^m
z_l^{\alpha_{i,j,l}}= [x_i,x_j] =[\hat{k}, x_i]=[h_1^{c_1} \cdots h_m^{c_m},
x_i].$$
Again, by Lemma \ref{lem20}, it follows that $\alpha_{i,j,l}$,  $1 \le i,j \leq
m$, $1 \le l \le 2m$, are uniquely determined by the structure constants of
${\rm U}_3(p^m)$.

That $\beta_{i,j,l}$,  $1 \le i,j \leq m$, $1 \le l \le
2m$, are uniquely determined  by the structure constants of
${\rm U}_3(p^m)$, follows on the same lines.
\end{proof}
Before proceeding for the proof of Theorem \ref{se5-thm1}, we summarize the
preceding discussion of this section.

For simplicity, fix a prime $p>2$, an
integer $m\ge 2$, and a finite $p$-group $H$ of
nilpotency class $3$ and of conjugate type $(1,p^{2m})$. Then, there exist
\begin{align*}
 \alpha_{i,j,l}, ~~ \beta_{i,j,l},~~
\gamma_{i,j,l},~~\delta_{i,j,l}, \lambda_{i,j,l},~~\mu_{i,j,l}, ~~\epsilon_{
i,l}, ~~ \nu_{i,l} \in\mathbb{F}_p \hskip5mm (1\le i,j\le m, ~~ 1\le l\le 2m),
\end{align*}
such that $H$ is isoclinic to the group $G$ with presentation as in Lemma
\ref{lem10}. By Theorems \ref{lem14}, \ref{se5-thm4} and \ref{se5-thm5}, 
$\alpha_{i,j,l}, ~~ \beta_{i,j,l},~~
\gamma_{i,j,l},~~\delta_{i,j,l}, \lambda_{i,j,l},~~\mu_{i,j,l},$ are uniquely
determined.
%
%
In particular, the isoclinism type of the group $H$ depends only on
$\epsilon_{i,l}$ and $\nu_{i,l}$, $1\le i,j\le m, ~~ 1\le l\le 2m$.

Thus, for the proof of Theorem \ref{se5-thm1}, we fix those unique 
$$\alpha_{i,j,l}, ~~ \beta_{i,j,l},~~
\gamma_{i,j,l},~~\delta_{i,j,l}, \lambda_{i,j,l},~~\mu_{i,j,l} \in
\mathbb{F}_p \hskip5mm (1\le i,j\le m, ~~ 1\le l\le 2m),$$
for which, a group given by the presentation as in Lemma \ref{lem10} is of conjugate
type $(1,p^{2m})$, for some $\epsilon$'s and $\nu$'s in $\mathbb{F}_p$. Then we
denote the resulting group with presentation as in Lemma \ref{lem10} by
$G(\epsilon,\nu)$.

\begin{lemma}\label{lem-final}
For any choice of $\epsilon_{i,l}, \nu_{i,l}, \epsilon_{i,l}',\nu_{i,l}' \in
\mathbb{F}_p$,
$1\le i\le m, ~~ 1\le l\le 2m$, the groups $G(\epsilon,\nu)$ and
$G(\epsilon',\nu')$ are isoclinic.
\end{lemma}
\begin{proof}
Consider the
presentation of $G(\epsilon,\nu)$  as in Lemma \ref{lem10}. To distinguish the generators of $G(\epsilon,\nu)$ and $G(\epsilon',\nu')$, we write the presentation of $G(\epsilon',\nu')$ as in Lemma \ref{lem10}, where we replace 
 $x_i$ by $\hat{x}_i$, $y_i$ by $\hat{y}_i$, $h_i$ by $\hat{h}_i$, and $z_l$ by
$\hat{z}_l$ for $1\le i\le m$, $1\le l\le 2m$. For simplicity, we denote the groups $G(\epsilon,\nu)$ and
$G(\epsilon',\nu')$ by $G_1$ and $G_2$ respectively. 

It follows, from the construction of $G_1$ and $G_2$, that  the map
$$x_i\Z(G_1)\mapsto \hat{x}_i\Z(G_2), \hskip3mm y_i\Z(G_1)\mapsto
\hat{y}_i\Z(G_2),
\hskip3mm h_i\Z(G_1)\mapsto \hat{h}_i\Z(G_2)$$
extends to an isomorphism $\phi:G_1/\Z(G_1)\rightarrow
G_2/\Z(G_2)$. Since $G_1'$ and $G_2'$ are elementary abelian, it is clear that
the map
$$h_i\mapsto \hat{h}_i, \hskip3mm z_l\mapsto \hat{z}_l$$
extends to an isomorphism $\theta:G_1'\rightarrow G_2'$.
Consider the diagram
\[
 \begin{CD}
   \overline G_1 \times \overline G_1  @>a_{G_1}>> G_1'\\
   @V{\phi\times\phi}VV        @VV{\theta}V\\
   \overline G_2 \times \overline G_2 @>a_{G_2}>> G_2',
  \end{CD}
\]
where $a_{G_1}$ and $a_{G_2}$ are the commutation maps as defined in Section
2.

From the commutator relations of $G_1$ and $G_2$ (that is, the relations
(R1)-(R9) of
$G_1$, and correspondingly those of $G_2$), it follows that the above
diagram commutes for the generators of $G_1$ and $G_2$ taken in their
presentation. A routine calculation now shows that the diagram commutes. 
\end{proof}

Stitching all the above pieces together, we get
\vskip3mm\noindent
{\it Proof of Theorem \ref{se5-thm1}.} For a prime $p >2$ and integer
$n=2m\geq 2$,
let $H$ be a finite $p$-group of nilpotency class $3$ and of conjugate type
$(1,p^{n})$. If $m=1$, then the result follows from
 \cite[Theorem 4.2]{Ishikawa1999}. Thus, we can assume that  $m\geq 2$.

By Lemma \ref{lem10}, there exist $\alpha$'s, $\beta$'s, $\gamma$'s,
$\delta$'s, $\lambda$'s, $\mu$'s, $\epsilon$'s and $\nu$'s in $\mathbb{F}_p$
such that $H$ is isoclinic to a group $G$ with the presentation as in Lemma
\ref{lem10}. By Theorems \ref{lem14}, \ref{se5-thm4} and \ref{se5-thm5},
$\alpha$'s, $\beta$'s, $\gamma$'s, $\delta$'s, $\lambda$'s, and $\mu$'s are
uniquely determined by the structure constants of ${\rm U}_3(p^m)$.
By Lemma \ref{lem-final}, the isoclinism type of $G$ is independent of the
choice of $\epsilon$'s and $\nu$'s in $\mathbb{F}_p$.

Thus, for any $m\geq 1$, $H$ is uniquely determined up to isoclinism, and
hence is isoclinic to the group $\mathcal{H}_m/\Z(\mathcal{H}_m)$ (see Section \ref{example}).
\hfill$\square$
\vskip3mm
\noindent {\bf Acknowledgement.} We thank Prof. C. M. Scoppola for letting us
know the examples of finite $p$-groups of nilpotency class $3$ and conjugate
type $(1, p^{2m})$ constructed in \cite{Dark-Scoppola}.

\end{document}